\documentclass[11pt]{amsart}
\usepackage{amssymb,mathrsfs,amsmath}
\usepackage{graphicx}
\usepackage{amscd,amsmath,amsopn,amssymb,amsthm,multicol}
\usepackage[color,matrix, all, 2cell]{xy}
\usepackage{amscd}
\usepackage{lscape}
\usepackage{slashed}
\usepackage{graphicx}
\usepackage{setspace}
\usepackage{upgreek}
\usepackage{textgreek}
\usepackage{enumerate}
\usepackage{color}
\usepackage{lscape}
\usepackage{tikz}
\usepackage{multirow}
\usepackage{cancel}
\usepackage{soul}
\usepackage{comment}
\usepackage{wasysym}
\usepackage{mathrsfs}
\usepackage{young}
\usepackage{mathtools}
\usepackage{bbm}
\usepackage{marginnote}

\numberwithin{equation}{section}
\newcommand{\modules}{$\text{Proposition 1.7}$}
\newcommand{\invarena}{$\text{Lemma 1.2}$}

\newcommand{\basSKEW}{$\text{Definition 1.24}$}
\newcommand{\EHbases}{$\text{Definition 1.17}$}
\newcommand{\genCa}{$\text{Proposition 3.14}$}
\newcommand{\algebtypes}{$\text{Corollary 3.15}$}
\newcommand{\genCar}{$\text{Corollary 3.17}$}
\newcommand{\zeroprog}{$\text{Theorem 4.3}$}
\newcommand{\symtorscomp}{$\text{Proposition 4.5}$}
\newcommand{\starconnection}{$\text{Theorem 3.8}$}
\newcommand{\starconnectionsp}{$\text{Theorem 3.12}$}

\newcommand{\darbas}{$\text{Definition A.7}$}
\newcommand{\kernel}{$\text{Lemma 3.11}$}
\newcommand{\gradbas}{$\text{Proposition A.9}$}
\newcommand{\homogthem}{$\text{Theorem 5.2}$}
\newcommand{\firstjet}{$\text{Proposition 4.8}$}

\newcommand{\R}{\mathbb{R}}
\newcommand{\Z}{\mathbb{Z}}
\newcommand{\C}{\mathbb{C}}
\newcommand{\Hn}{\mathbb{H}}
\newcommand{\p}{\mathfrak{p}}

\newcommand{\g}{\mathfrak{g}}
\newcommand{\gl}{\mathfrak{gl}}
\renewcommand{\sl}{\mathfrak{sl}}
\renewcommand{\so}{\mathfrak{so}}
\renewcommand{\sp}{\mathfrak{sp}}

\renewcommand{\sc}{\mathsf{sc}}

\renewcommand{\Im}{\mathbbm{Im}}
\DeclareMathOperator{\imm}{\mathsf{Im}}

\DeclareMathOperator{\Tor}{\mathsf{Tor}}
\DeclareMathAlphabet{\mathscrbf}{OMS}{mdugm}{b}{n}
\DeclareMathOperator{\Spin}{\mathsf{Spin}}
\DeclareMathOperator{\SO}{\mathsf{SO}}
\DeclareMathOperator{\Sp}{\mathsf{Sp}}
\DeclareMathOperator{\SU}{\mathsf{SU}}
\DeclareMathOperator{\hor}{\mathsf{hor}}
\DeclareMathOperator{\pul}{\mathsf{pul}}

\DeclareMathOperator{\U}{\mathsf{U}}

\DeclareMathAlphabet{\mathpzc}{OT1}{pzc}{m}{it}
\DeclareMathOperator{\Hh}{\mathsf{H}}
\DeclareMathOperator{\Oo}{\mathsf{O}}
\DeclareMathOperator{\E}{\mathsf{E}}
\DeclareMathOperator{\Gl}{\mathsf{GL}}
\DeclareMathOperator{\Sl}{\mathsf{SL}}

\DeclareMathOperator{\Aut}{\mathsf{Aut}}
\DeclareMathOperator{\id}{\mathsf{id}}
\DeclareMathOperator{\w}{\mathsf{w}}
\DeclareMathOperator{\Ss}{S}
\DeclareMathOperator{\Ed}{\mathsf{End}}

\DeclareMathOperator{\J}{\mathsf{J}}

\DeclareMathOperator{\K}{\mathsf{K}}
\DeclareMathOperator{\Ad}{\mathsf{Ad}}
\DeclareMathOperator{\vol}{\mathsf{vol}}
\DeclareMathOperator{\Id}{\mathsf{Id}}
\newcommand{\fr}{\mathfrak}
\newcommand{\al}{\alpha}

\newcommand{\mc}{\mathcal}

\newcommand{\cc}{\big(}
\newcommand{\CC}{\Big(}
\newcommand{\rr}{\big)}
\newcommand{\RR}{\Big)}

\DeclareMathAlphabet{\mathscrbf}{OMS}{mdugm}{b}{n}

\DeclareMathOperator{\Tr}{\mathsf{Tr}}

\DeclareMathOperator{\ke}{\mathsf{Ker}}

\DeclareMathOperator{\dd}{d}

\newtheorem{theorem}{Theorem}[section]
\newtheorem{lem}[theorem]{Lemma}
\newtheorem{prop}[theorem]{Proposition}
\newtheorem{corol}[theorem]{Corollary}
\newtheorem{rem}[theorem]{Remark}

\theoremstyle{definition}
\newtheorem{defi}[theorem]{Definition}
\newtheorem{example}[theorem]{Example}

\theoremstyle{remark}

\numberwithin{equation}{section}

\definecolor{dark}{rgb}{0.18,0.18,0.68}
\definecolor{mydark}{rgb}{0.78,0.08,0.08}
\definecolor{crew}{rgb}{0.2,0.5,0.2}
\definecolor{mmg}{rgb}{0.31,0.50,0.23}
\definecolor{dblue}{rgb}{0.01,0.01,0.44}
\definecolor{red}{rgb}{0.57,0.11,0.15}
\definecolor{cobalt}{RGB}{61,89,171}
\usepackage[colorlinks,citecolor=cobalt,linkcolor=cobalt,urlcolor=cobalt,pdfpagemode=UseNone,backref = page]{hyperref}

\input ulem.sty

\title[Differential geometry of $\SO^\ast(2n)$-type structures - Integrability]{
Differential geometry of $\SO^\ast(2n)$-type structures - Integrability}

\author{Ioannis Chrysikos} 
\address{Faculty of Science, University of Hradec Kr\'alov\'e, Rokitanskeho 62, Hradec Kr\'alov\'e
50003, Czech Republic}
\email{ioannis.chrysikos@uhk.cz}

\author{Jan Gregorovi\v c} 
\address{Faculty of Science, University of Hradec Kr\'alov\'e, Rokitanskeho 62, Hradec Kr\'alov\'e
50003, Czech Republic}
\email{jan.gregorovic@seznam.cz}

\author{Henrik Winther} 
\address{Department of Mathematics and Statistics, Masaryk University, Kotl\'a\v{r}sk\'a 2, Brno 611 37, Czech Republic} 
\email{winther@math.muni.cz}

\begin{document}

\begin{abstract}
We study almost hypercomplex skew-Hermitian structures and almost quaternionic skew-Hermitian structures, as the geometric structures underlying $\SO^*(2n)$- and $\SO^*(2n)\Sp(1)$-structures, respectively. 
The corresponding intrinsic torsions were computed in the previous article in this series, and the algebraic types of the geometries were derived, together with the minimal adapted connections (with respect to certain normalizations conditions). 
Here we use these results to present the related first-order integrability conditions in terms of the algebraic types and other constructions. 
In particular, we use distinguished connections to provide a more geometric interpretation of the presented integrability conditions and highlight some features of certain classes. 
The second main contribution of this note is the illustration of several specific types of such geometries via a variety of examples. 
We use the bundle of Weyl structures and describe examples of $\SO^*(2n)\Sp(1)$-structures in terms of functorial constructions in the context of parabolic geometries. 
\end{abstract}

\maketitle

\tableofcontents


\section*{Introduction}\label{intro}
This is the second part in a series of articles devoted to the study of geometric structures underlying manifolds admitting a reduction of the frame bundle to one of the Lie groups $\SO^*(2n)$ or $\SO^*(2n)\Sp(1)$. The group $\SO^*(2n)$ denotes the quaternionic real form of $\SO(2n, \C)$, while $\SO^*(2n)\Sp(1)=\SO^*(2n)\times_{\Z_2}\Sp(1)$. 
We call these structures \textsf{almost} \textsf{hypercomplex skew-Hermitian structures} and \textsf{almost quaternionic skew-Hermitian structures} respectively. 

In the first part \cite{CGWPartI} we describe such structures in terms of pairs $(H, \omega)$ and $(Q, \omega)$. 
Here, $H$ is an almost hypercomplex structure, $Q$ is an almost quaternionic structure and $\omega$ is a \textsf{scalar 2-form}, i.e., a non-degenerate 2-form which is $H$-Hermitian, respectively $Q$-Hermitian. 
To such pairs we may associate a symmetric 4-tensor $\Phi$ and a quaternionic skew-Hermitian form $h$. 
In \cite{CGWPartI} it is shown that both these tensors are stabilized by $\SO^*(2n)\Sp(1)$, and so each of them can serve as a defining tensor for $\SO^\ast(2n)\Sp(1)$-structures. 
This allows an alternative approach to almost quaternionic skew-Hermitian geometry, in particular, $\Phi$ is the analogue of the fundamental 4-form $\Omega$ in almost quaternionic Hermitian (qH) geometry (see for example \cite{AM}). 
Another achievement in \cite{CGWPartI} is the computation of the corresponding intrinsic torsion of $\SO^*(2n)$- or $\SO^*(2n)\Sp(1)$-structures, and the description of the corresponding (minimal) adapted connections, denoted by $\nabla^{H, \omega}$ and $\nabla^{Q, \omega}$, respectively. 
In addition, \cite{CGWPartI} provides the number of algebraic types of $\SO^*(2n)$- and $\SO^*(2n)\Sp(1)$-geometries. For $n>3$ there are seven special $\Sp(1)$-invariant pure types $\mc{X}_1, \ldots, \mc{X}_7$ of $\SO^*(2n)$-structures, and ten general types, so up to $2^{10}$ algebraic types of $\SO^*(2n)$-geometries. For $\SO^*(2n)\Sp(1)$ $(n>3)$ there are five pure types $\mc{X}_1, \ldots, \mc{X}_5$ and hence up to $2^5$ algebraic types of $\SO^*(2n)\Sp(1)$-geometries. Finally, for the low-dimensional cases $n=2, 3$, \cite{CGWPartI} shows that there exist further types of such non-integrable geometries.

In this paper we explore all these types and explain the contribution of the different intrinsic torsion components in the obstruction of the integrability of $H$, $Q$ and $\omega$. One of our main results is the description of 1st-order integrability conditions for both the $G$-structures under investigation. 
The methodology for such a procedure partially builds on the results from \cite{CGWPartI} and on branching rules. We also adopt more geometric approaches, given in terms of distinguished connections. This latter technique allows us to provide a more geometric interpretation of the presented integrability conditions and highlight some features of certain algebraic classes. Moreover, for $\SO^*(2n)\Sp(1)$-structures we establish a theory related to the covariant derivative $\nabla\Phi$, where $\nabla$ is any affine connection on $M$. Then we elaborate on this theory for an almost symplectic connection $\nabla^{\omega}$ (with respect to the related scalar 2-form $\omega$), and for the unimodular Oproiu connection $\nabla^{Q, \vol}$ (with respect to the pair $(Q, \vol=\omega^{2n})$), etc. Sections \ref{integrability} and \ref{sec2II} are devoted to the description of these methodologies and aforementioned outcome.

The other major contribution of this paper is the presentation of a variety of examples of manifolds admitting such $G$-structures, and in particular of examples realizing some of the aforementioned algebraic types.
The first class of examples is based on modifying global frames on $\R^{4n}$, or on an arbitrary almost symplectic manifold, and is presented in Section \ref{examples}.
The second class of examples is presented in Section \ref{homogeneousex}. 
Here we provide a general description of homogeneous almost hypercomplex skew-Hermitian manifolds, and almost quaternionic skew-Hermitian manifolds.
This includes a presentation of the invariant adapted connections, in terms of \textsf{Nomizu maps} (and the soldering form). We also derive how the intrinsic torsion arises from such a description. 
Then we analyze some explicit examples. In particular, we examine the reductive homogeneous space $M=K/L=\Sl(4,\R)/\Sl(2,\R)$ and show the existence of a $\Sl(4,\R)$-invariant $\SO^*(2n)$-structure of generic $\Sp(1)$-invariant type $\mc{X}_{1234567}$. 

The final class of examples is based on general functors from the categories of almost quaternionic manifolds, or quaternionic affine manifolds, and is presented in Section \ref{Weylss}.
In this direction, initially we focus on the total space of the \textsf{Weyl bundle} over the quaternionic projective space $N=\Hn {\sf P}^n = G/P$, or on the \textsf{cotangent bundle} of $N$, respectively. The related constructions are essentially motivated by the fact that for any quaternionic vector space $(V, Q)$, the contangent space $T^*V\cong V\times V^*$ admits a canonical linear quaternionic skew-Hermitian structure. In this case, the associated scalar 2-form $\omega$ is the canonical symplectic form given by the natural pairing $V\times V^*\to \R$ (see Section \ref{Weylss}). We show that there are two possible ways to generalize this result to the manifold setting, which are essentially are related to each other, since 
the Weyl bundle and cotangent bundle are diffeomorphic. When the source is a more general quaternionic affine manifold than $\Hn {\sf P}^n$, we compare the various possible combinations of the resulting almost quaternionic and almost symplectic structures. 
Then we describe conditions for their pairwise compatibility as almost quaternionic skew-Hermitian structures.

In contrast, note that there is no faithful functor from the category of almost symplectic structures to the category of almost hypercomplex/quaternionic skew-Hermitian structures. This is because hypercomplex/quaternionic morphisms are determined by finite jets and form a finite dimensional Lie group. However, symplectomorphisms are {\it not} determined by some finite jet, and their pseudo-group can be infinite dimensional.

In the final section we pose some open problems related to $\SO^*(2n)$- and $\SO^*(2n)\Sp(1)$-structures. 
Some of these emphasize further open tasks which are related to the local differential geometry of such $G$-structures (curvature invariants, twistor constructions, a metric view point of $\SO^*(2n)$-structures, etc). 
Other questions relate to a clarification of the relationship between $\SO^*(2n)\Sp(1)$-structures and quaternionic geometries in terms of parabolic geometries.
We treat some of these questions and further related tasks in the third part of this series of works, see \cite{CGWPartIII}. Finally, the current article also includes an appendix where we describe some basic topological features of manifolds admitting $\SO^*(2n)$- or $\SO^*(2n)\Sp(1)$-structures.

\medskip
\noindent {\bf Acknowledgments:}
I.C. thanks Masaryk University for hospitality. H.W. acknowledges full support from the Czech Science Foundation via the project (GA\v{C}R) No. GX19-28628X.

\section{Integrability conditions of $\SO^*(2n)$-structures and $\SO^*(2n)\Sp(1)$-structures}\label{integrability}

\subsection{Basics on almost hypercomplex/quaternionic skew-Hermitian structures}
We begin by briefly recalling the concept of almost hypercomplex skew-Hermitian structures, and almost quaternionic skew-Hermitian structures, as introduced in \cite{CGWPartI}. Let us consider a $4n$-dimensional connected smooth manifold $M$ and assume, once and for all, that $n>1$. Recall that:\\
$\bullet$ \ An \textsf{almost hypercomplex structure} $H=\{J_a : a=1, 2, 3\}$ on $M$ is a triple of almost complex structures satisfying the quaternionic relations. Then, $(M, H)$ is called an \textsf{almost hypercomplex manifold}. 
A covariant 2-tensor $F$ on $(M, H)$ is called \textsf{$H$-Hermitian} if $F(J_{a}X, J_{a}Y)=F(X, Y)$ for any $X, Y\in\Gamma(TM)$ and $a=1, 2, 3$. \\
$\bullet$ \ An \textsf{almost hypercomplex skew-Hermitian} (hs-H for short) structure on $M$ consists of a pair $(H, \omega)$, where $H$ is an almost hypercomplex structure and $\omega\in\Omega^{2}(M)$ is a non-degenerate 2-form which is $H$-Hermitian. 
Then, $(M, H, \omega)$ is called an \textsf{almost hypercomplex skew-Hermitian manifold}, and $\omega$ is called a \textsf{scalar 2-form with respect to $H$}. \\
$\bullet$ \ An \textsf{almost quaternionic structure} $Q\subset\Ed(TM)$ on $M$ is a rank $3$ sub-bundle locally spanned by an almost hypercomplex structure. This means that there exists a locally defined triple $H=\{J_{a} : a=1, 2, 3\}$ on some open neighbourhood $U\subset M$, called an \textsf{admissible basis} of $Q$, such that $Q_{x}=\langle H_{x} \rangle$, for any $x\in U$. Then, $(M, Q)$ is called an \textsf{almost quaternionic manifold}. 
A covariant 2-tensor $F$ on $(M, Q)$ is called \textsf{$Q$-Hermitian} if $F(J_aX, J_aY)=F(X, Y)$, for any $X, Y\in\Gamma(TM)$ and any admissible basis $\{J_a : a=1, 2, 3\}$ of $Q$. \\ 
$\bullet$ \ An \textsf{almost quaternionic skew-Hermitian} (qs-H for short) structure on $M$ consists of a pair $(Q , \omega)$, where $Q$ is an almost quaternionic structure and $\omega\in\Omega^{2}(M)$ is a non-degenerate 2-form which is $Q$-Hermitian. Then, $(M, Q, \omega)$ is called an \textsf{almost quaternionic skew-Hermitian manifold} and $\omega$ is called a \textsf{scalar 2-form with respect to $Q$}.

Let us now recall some of our conventions from Section 1 in \cite{CGWPartI}, in terms of the known $\E\Hh$-formalism. So, let $\E=\C^{2n}$ and $\Hh=\C^2$ be the standard representations of $\SO^*(2n)$ and $\Sp(1)$ respectively. They are both of quaternionic type. We denote by $[\E\Hh]$ the $4n$-dimensional real vector space equipped with the standard (linear) quaternionic structure $Q_0=\sp(1)\cong [S^2\Hh]$.
The module $[\E\Hh]$ is a real subspace in $\E\otimes_\C \Hh$, and elements of $[\E\Hh]$ will be often denoted by 
\[
a:= \left( 
	\begin{smallmatrix}
		a \\
		0
	\end{smallmatrix}
\right)\otimes \left( 
	\begin{smallmatrix}
		1 \\
		0
\end{smallmatrix}
\right)+\left( 
	\begin{smallmatrix}
		0 \\
		\bar{a}
	\end{smallmatrix}
\right)\otimes \left( 
	\begin{smallmatrix}
		0 \\
		1
	\end{smallmatrix}
\right)\,,\quad 
bj:= \left( 
	\begin{smallmatrix}
		0 \\
		-b
\end{smallmatrix}
\right)\otimes \left( 
	\begin{smallmatrix}
		1 \\
		0
	\end{smallmatrix}
\right)+ \left( 
	\begin{smallmatrix}
		\bar{b}\\
		0
	\end{smallmatrix}
\right)\otimes \left( 
	\begin{smallmatrix}
		0 \\
		1
	\end{smallmatrix}
\right)\,.
\]

Almost hs-H structures are in bijection with reductions $\mc{P}$ of the frame bundle $\mc{F}(M)$ of $M$ to $\SO^*(2n)$. Such reductions consist of skew-Hermitian bases of $T_{x}M$ with respect to $(H_{x}, \omega_x)$, inducing a linear hypercomplex isomorphism $u : T_{x}M\to [\E\Hh]$.
We say that a basis $e_1,\dots, e_{2n},f_1,\dots, f_{2n}$ is \textsf{adapted} to the linear hypercomplex structure $H_x=\{J_1, J_2, J_3\}$ on $T_{x}M$ if 
\[
J_1(e_{c})=e_{c+n}\,,\quad
J_2(e_{c})=f_{c}\,,\quad 
J_3(e_{c})=f_{c+n}\,,
\] 
for all $1\leq c\leq n$.
We should mention that such a basis differs from an \textsf{adapted basis (or frame) of $H$} in terms of \cite[p.~209]{AM}, but they are related by a permutation. 
Recall also that by a \textsf{skew-Hermitian basis} with respect to $(H_x, \omega_x)$, we mean 
a symplectic basis $(e_1,\dots, e_{2n},f_1,\dots, f_{2n})$ 
of the scalar 2-form $\omega_x$, i.e., 
\[
	\omega(e_r,e_s)=0\,,\quad \omega(f_r,f_s)=0\,,\quad \omega(e_r,f_r)=1\,,\quad\omega(e_r,f_s)=0\,, (r\neq s)
\]
for $1\leq r\leq 2n$ and $1\leq s\leq 2n$, which is also adapted to $H$ in the above sense, see also \cite[\basSKEW]{CGWPartI}. 
If $(u_1,\dots,u_{2n},v_1,\dots,v_{2n})^t$ are coordinates in the skew-Hermitian basis, then the isomorphism $u$ is given by
\[
	a=(u_1+iu_{n+1},\dots,u_n+iu_{2n})^t\,,\quad bj=(v_1+iv_{n+1},\dots,v_n+iv_{2n})^t.
\]
In this way, the almost hypercomplex structure $H=\{J_1, J_2, J_3\}$ corresponds to a unique admissible basis $H_0$ of $Q_0$ on $[\E\Hh]$. Moreover, the almost symplectic form $\omega$ corresponds to a unique scalar 2-form $\omega_0$ on $[\E\Hh]$. We refer to them by the terms \textsf{standard admissible basis} of $Q_0$ on $[\E\Hh]$ and \textsf{standard scalar 2-form} on $[\E\Hh]$, respectively.

Similarly, almost qs-H structures are in bijection with reductions $\mc{Q}$ of the frame bundle of $M$ to $\SO^*(2n)\Sp(1)$. Such reductions consist of all skew-Hermitian bases of $T_{x}M$ with respect to $(H_x, \omega_x)$, where $H_x$ is some admissible basis $H_{x}$ of $Q_x$. Similarly to the hypercomplex case, such bases induce a linear quaternionic isomorphisms $u : T_{x}M\to [\E\Hh]$. 

Let $\{J_{a} : a=1, 2, 3\}$ be a local admissible basis of $Q$ and let $(Q, \omega)$ be an almost qs-H structure on $M$. Then, we may attach three pseudo-Riemannian metric tensors of signature $(2n, 2n)$, given by $g_{J_a}(X, Y):=\omega(X, J_{a}Y)$. 
In terms of a skew-Hermitian frame $(e_1,\dots,e_{2n},f_1,\dots,f_{2n})$ of $TM$ and its dual coframe $(e_1^*,\dots,e_{2n}^*,f_1^*,\dots,f_{2n}^*)$, we have
\begin{align*}
g_{J_1}&=\sum_{a=1}^n (e_a^*\odot f_{n+a}^*-e_{n+a}^*\odot f_{a}^*)\,,\\
g_{J_2}&=\sum_{a=1}^n (-e_a^*\otimes e_{a}^*-f_{a}^*\otimes f_{a}^*+e_{n+a}^*\otimes e_{n+a}^*+f_{n+a}^*\otimes f_{n+a}^*)\,,\\
g_{J_3}&=\sum_{a=1}^n (-e_a^*\odot e_{a+n}^*-f_{a+n}^*\odot f_{a}^*)\,.
\end{align*}
The tensors $g_{J_{a}}$ $(a=1, 2, 3)$ are globally defined only when $H$ is a global admissible basis, i.e., when $(H, \omega)$ is an almost hs-H structure on $M$. However, the \textsf{quaternionic skew-Hermitian form $h$} defined by
\[
h(X,Y)Z:=\omega(X,Y)Z+\sum_{a} g_{J_a}(X, Y)J_aZ\,,\quad X, Y, Z\in\Gamma(TM)\,,
\]
is independent of the local admissible basis $H$ of $Q$, and moreover has stabilizer the Lie group $\SO^*(2n)\Sp(1)$. Thus, it can be viewed as 
a $(1, 3)$-tensor globally defined on $(M, Q, \omega)$. Similarly, 
the symmetric 4-tensor 
\[
\Phi:=\sum_{a}g_{J_{a}}\odot g_{J_{a}}=\mathsf{Sym}\big(\omega(\cdot\,, \Im(h)\cdot)\big)
\] 
is independent of the local admissible basis $H$ and has stabilizer $\SO^*(2n)\Sp(1)$, see \cite{CGWPartI} for more details. Thus, it is a global tensor on $(M, Q, \omega)$, i.e., $\Phi\in\Gamma(S^4T^*M)$. As a consequence, these tensors provide an alternative approach to almost qs-H structures. In terms of the skew-Hermitian (co)-frame described above, we have
{\footnotesize
\begin{eqnarray*}
\Phi&=&(\sum_{a=1}^n (e_a^*\odot f_{n+a}^*-e_{n+a}^*\odot f_{a}^*)\odot \sum_{a=1}^n (e_a^*\odot f_{n+a}^*-e_{n+a}^*\odot f_{a}^*))\\
&&+(\sum_{a=1}^n (-e_a^*\otimes e_{a}^*-f_{a}^*\otimes f_{a}^*+e_{n+a}^*\otimes e_{n+a}^*+f_{n+a}^*\otimes f_{n+a}^*))\\
&&\odot (\sum_{a=1}^n (-e_a^*\otimes e_{a}^*-f_{a}^*\otimes f_{a}^*+e_{n+a}^*\otimes e_{n+a}^*+f_{n+a}^*\otimes f_{n+a}^*))\\
&&+(\sum_{a=1}^n (-e_a^*\odot e_{a+n}^*-f_{a+n}^*\odot f_{a}^*))\odot (\sum_{a=1}^n (-e_a^*\odot e_{a+n}^*-f_{a+n}^*\odot f_{a}^*))\,.
\end{eqnarray*}}

Let us finally recall that on an almost hypercomplex/quaternionic skew-Hermitian manifold we have introduced (minimal) adapted connections $\nabla^{H, \omega}$ and $\nabla^{Q, \omega}$, respectively. In Section \ref{sec2II} we will recall further related details and provide a more geometric viewpoint of these adapted connections.

\subsection{Applications of branching rules}
In this section we use the results presented in the first article \cite{CGWPartI} on intrinsic torsion and minimal connections of $\SO^*(2n)\Sp(1)$- and $\SO^*(2n)$-structures, to derive 1st-order integrability conditions for such structures. In particular, by \cite{CGWPartI} we shall use \genCa, \algebtypes, \genCar, \zeroprog, and \symtorscomp. Some of our arguments are also based on ``\textsf{branching rules}'' and in particular on the following general remark from the theory of $G$-structures.
\begin{rem}\label{genHenrik}\textnormal{
Let $K\subset G$ be a reductive Lie subgroup of a reductive linear Lie group $G$. Then, a $K$-structure on a manifold $M$ is also an example of a $G$-structure on $M$. 
Assume that we have equivalence classes of minimal $K$-connections and minimal $G$-connections. 
Let $\nabla^K$ and $\nabla^G$ be minimal connections for the $K$-structure and $G$-structure, respectively. Then we have
\begin{equation}
	\nabla^K = \nabla^G + A
\end{equation}
for some (1, 2)-tensor field $A\in \Gamma(T^\ast M \otimes \fr{g})$, where $\fr{g}$ is the Lie algebra of $G$. Therefore, the restriction to $K$ of the intrinsic torsion module of $G$ will appear as a submodule of the intrinsic torsion module of $K$. 
Consequently, one can obtain geometric interpretations of certain submodules of the $K$-intrinsic torsion by branching the intrinsic torsion modules of $G$ to $K$.}
\end{rem}
Next we shall apply this method to $\SO^*(2n)\Sp(1)$-structures, and also to $\SO^*(2n)$-structures. 
In the latter case much of our focus is devoted to the types $\mc{X}_1, \ldots, \mc{X}_7$, characterized in terms of $\Sp(1)$-invariant conditions (see \cite[\genCar]{CGWPartI}).

\subsection{1st-order integrability conditions for $\SO^{*}(2n)\Sp(1)$-structures}\label{mainthems}
Initially, it is convenient to derive 1st-order integrability conditions about $\SO^{*}(2n)\Sp(1)$-structures.
\subsubsection{Quaternionic torsion}
Let us consider the Lie group $G:=\Gl(n, \Hn)\Sp(1)$. Then, $K:=\SO^\ast(2n)\Sp(1)\subset G$ is a (closed) Lie subgroup of $G$. The intrinsic torsion module of $G$ was computed in \cite{Salamon86} (see also \cite{AM}). They found a unique normalization condition (due to multiplicity one in the decomposition of quaternionic torsion tensors), given by
\begin{equation*}
	\mc{D}\big(\gl(n, \Hn)\oplus\fr{sp}(1)\big)= R(3\theta + \pi_1 + 2\pi_{2n})\,.
\end{equation*}
Here $R(3\theta + \pi_1 + 2\pi_{2n})$ denotes the $\fr{gl}(n, \Hn)\sp(1)$-module with highest weight $3\theta + \pi_1 + 2\pi_{2n}$, where $\theta$ is the fundamental weight of $\sp(1)$, and $\pi_i$ is the $i$-th fundamental weight of $\fr{sl}(n, \Hn)$. The identity matrix acts as scalar multiplication by $-1$, which fixes the action of the center.
Hence, any almost quaternionic connection has a $G$-equivariant decomposition of its torsion tensor with one component of this isotype. This does not require that the connection is minimal, and the component may vanish.
Therefore, this applies in particular to connections which are minimal with respect to $K$. 
Thus we may branch to $K$, a procedure which gives rise to the following
\begin{prop}
Viewed as a $\SO^\ast(2n)\Sp(1)$-module, the restriction $\mc{D}\big(\gl(n, \Hn)\oplus\fr{sp}(1)\big)|_{\SO^\ast(2n)\Sp(1)}$ satisfies the following	\begin{equation}
		\mc{D}\big(\gl(n, \Hn)\oplus\fr{sp}(1)\big)|_{\SO^\ast(2n)\Sp(1)} \cong
		[\K S^3\Hh]^*\oplus [\Lambda^3\E S^3\Hh]^*=\mc{X}_{12}=\mc{X}_1\oplus\mc{X}_2\,.
	\end{equation}
\end{prop}
Therefore, the components of the intrinsic torsion of an almost qs-H structure taking values in these submodules only depends on, and coincides with, the intrinsic torsion of the underlying almost quaternionic structure. This yields the following.
\begin{theorem}\label{quaternionicthm}
	Let $(Q,\omega)$ be an almost qs-H structure with canonical connection $\nabla^{Q,\omega}$. Then $Q$ is quaternionic, if and only if $(Q, \omega)$ is of type $\mc{X}_{345}$, that is the torsion components $[\K S^3\Hh]^*$ and $[\Lambda^3\E S^3\Hh]^*$ of $\nabla^{Q,\omega}$ vanish.
\end{theorem}
\subsubsection{Symplectic torsion}
Assume that $G:=\Sp(4n,\R)$. It is well known that (see for example \cite{APick, Cap})
\begin{equation}
	\mc{D}(\fr{sp}(4n,\R)) \cong \Lambda^3 (\R^{4n})^\ast\cong \Lambda^3_{0} (\R^{4n})^\ast\oplus (\R^{4n})^\ast\,,
\end{equation}
and in particular the intrinsic torsion can be identified with the 3-form $\dd\omega$.
Similarly to the previous case, we may branch this module with respect to $\so^\ast(2n)\oplus\fr{sp}(1)$. 
\begin{prop}\label{spnormal}
Viewed as a $\SO^\ast(2n)\Sp(1)$-module,	the restriction $\mc{D}(\sp(4n,\R))|_{\SO^\ast(2n)\Sp(1)} $ satisfies the following 
\begin{equation}\label{spnormal2}
	\mc{D}(\sp(4n,\R))|_{\SO^\ast(2n)\Sp(1)} \cong
	[\Lambda^3\E S^3\Hh]^*\oplus [\K\Hh]^* \oplus [\E\Hh]^*=\mc{X}_{234}\,.
\end{equation}
\end{prop}
Thus, a combination of Proposition \ref{spnormal} with \cite[\symtorscomp]{CGWPartI} yields the following theorem.
\begin{theorem}\label{symplecticthem}
	Let $(Q,\omega)$ be an almost qs-H structure with canonical connection $\nabla^{Q,\omega}$. Then $\omega$ is symplectic, if and only if $(Q, \omega)$ is of type $\mc{X}_{15}$, that is the torsion components $[\Lambda^3\E S^3\Hh]^*$, $[\K\Hh]^*$ and $[\E\Hh]^*$ of $\nabla^{Q,\omega}$ vanish.
\end{theorem}
\begin{rem}\label{confchang}
	\textnormal{Similarly to the almost-symplectic case, for $\SO^*(2n)\Sp(1)$-structures the vectorial torsion component with values in $\mc{X}_{4}=[\E\Hh]^*$ can be modified by a conformal change 
\[
	\omega \mapsto f\cdot \omega\,,
\]
	where $f\in C^\infty (M)$ is some non-vanishing function. Then, by \cite[\starconnectionsp]{CGWPartI} we deduce that the canonical connection $\nabla^{Q,\omega}$ is modified by 
\[
	 -\dd f\otimes \id-\frac{4n}{n+1} \pi_S(\omega\otimes \dd f^T)+\frac{n}{n+1}\sum_{a=1}^3 \dd f \circ J_a\otimes J_a\,,
\]
	where $\dd f^{T}$ denotes the symplectic transpose of $\dd f$ and $\pi_{S}$ is the projection defined by 
\[
\pi_S(\omega\otimes Z)(X,Y):=Sym\Big(\pi_{1, 1}(\omega(X,.) \otimes Z)\Big)Y\,,
\]
with 
\[
\pi_{1, 1} : \gl([\E\Hh])\to \gl(n,\Hn)\,,\quad \pi_{1, 1}(\omega(X,.) \otimes Z)Y=\frac{1}{4}\Big(\omega(X,Y) Z-\sum_a g_{J_a}(X,Y) J_aZ\Big)\,,
\]
see \cite[Section 3.2]{CGWPartI}. It turns out that this change does not affect the compatibility with $Q$. In this way we obtain a new $\SO^*(2n)\Sp(1)$-structure $(Q, f\cdot\omega)$. When the initial $\SO^*(2n)\Sp(1)$-structure $(Q, \omega)$ is of type $\mc{X}_{145}$, then (locally) there is a smooth function $f$ on $M$ such that $f\cdot\omega$ is symplectic. This is equivalent to saying that $(Q, f\cdot\omega)$ is of type $\mc{X}_{15}$. We deduce that the type $\mc{X}_{145}$ of $\SO^*(2n)\Sp(1)$-structures characterizes a certain subclass of the so-called \textsf{parabolic (locally) conformally symplectic structures}, examined by \v{C}ap and Sala\v{c} in \cite{Cap}.}
\end{rem}
\subsubsection{Intersection torsion and compatibility torsion}
Notice that the submodule $\mc{X}_2=[\Lambda^3\E S^3\Hh]^*$ appears in both the quaternionic and the symplectic torsion. 
Since $\mc{X}_2$ has multiplicity one in $\mc{D}(\gl(n,\Hn)\oplus\fr{sp}(1))$, $\mc{D}(\fr{sp}(4n,\R))$ and $\Lambda^2[\E\Hh]^*\otimes [\E\Hh]$, this submodule must coincide. This represents a smooth compatibility condition between those almost symplectic structures and almost quaternionic structures which are algebraically compatible. Thus we may call the submodule $\mc{X}_2$ the module corresponding to \textsf{intersection torsion}. 
On the other side, and since we have exhausted the quaternionic and symplectic torsion, we now consider the complement of their union in the intrinsic torsion. This submodule will be called the module corresponding to \textsf{compatibility torsion}, and it consists of the simple submodule $\mc{X}_5=[S^3_0\E \Hh]^*$. 
By \cite[\genCa]{CGWPartI}, this submodule is isotypically unique in the space of torsion tensors, hence it is independent of our normalization condition. From this and the previous two theorems, we obtain
\begin{theorem}
	Let $(Q,\omega)$ be an almost qs-H structure with canonical connection $\nabla^{Q,\omega}$. Then $Q$ is quaternionic and $\omega$ symplectic, if and only if $(Q, \omega)$ is of type $\mc{X}_5$, that is the torsion of $\nabla^{Q,\omega}$ is contained in the submodule $[S^3_0\E \Hh]^*$.
\end{theorem}


\subsection{1st-order integrability conditions for $\SO^*(2n)$-structures} \label{secintegrable}
Next we will determine 1st-order integrability conditions for $\SO^*(2n)$-structures on some fixed $4n$-dimensional smooth connected manifold $M$.
To do so we benefit from \cite[\genCar]{CGWPartI} and the interpretation of the intrinsic torsion module $\mc{H}(\fr{so}^*(2n))$ as a $\Sp(1)$-module. 

\subsubsection{Hypercomplex torsion}
In terms of Remark \ref{genHenrik} let us now set 
$G:=\SO^*(2n)\Sp(1)$ and $K:=\SO^\ast(2n)\subset G$. 
Then, by \cite[\zeroprog]{CGWPartI} we may pose the following decomposition of hypercomplex torsion
\[
\imm(\pi_H) \cong \mc{X}_1\oplus \mc{X}_2\oplus\mc{X}_6 \cong 2\Lambda^3 \E^*\oplus 2\K^*\oplus2\E^*\,.
\]
Here, the first splitting should be read in terms of $G$-modules, while the second one in terms of $K$-modules.

Therefore, the components of the intrinsic torsion as an almost hs-H structure taking values in these $G$-modules only depends on, and coincides with, the intrinsic torsion of the underlying almost hypercomplex structure. 
This yields the following:
\begin{theorem}\label{hypercomplexthm}
	Let $(H,\omega)$ be an almost hs-H structure with canonical connection $\nabla^{H,\omega}$. Then $H$ is hypercomplex, if and only if $(H, \omega)$ is of type $\mc{X}_{3457}$, that is the torsion components $\mc{X}_1=[\K S^3\Hh]^*$, $\mc{X}_2=[\Lambda^3\E S^3\Hh]^*$ and $\mc{X}_6=[\E S^3\Hh]^*$ of $\nabla^{H,\omega}$ vanish.
\end{theorem}
\subsubsection{Symplectic torsion}
Let us now apply Proposition \ref{spnormal} to obtain the following branching:
\begin{equation}\label{brancso}
		\mc{D}(\sp(4n,\R))|_{\SO^\ast(2n)} \cong	\mc{X}_{234}|_{\SO^*(2n)}=([\Lambda^3\E S^3\Hh]^*\oplus [\K\Hh]^* \oplus [\E\Hh]^*)|_{\SO^*(2n)}= 2\Lambda^3\E^*\oplus\K^*\oplus\E^*
\end{equation}
As a consequence of \cite[\symtorscomp]{CGWPartI} and Proposition \ref{spnormal} we obtain the following theorem.
\begin{theorem}\label{symplecticthem2}
	Let $(H,\omega)$ be an almost hs-H structure with canonical connection $\nabla^{H,\omega}$. Then $\omega$ is symplectic, if and only if $(H, \omega)$ is of type $\mc{X}_{1567}$, that is the torsion components $[\Lambda^3\E S^3\Hh]^*$, $[\K\Hh]^*$ and $[\E\Hh]^*$ of $\nabla^{H,\omega}$ vanish.
\end{theorem}

\subsubsection{Intersection torsion and compatibility torsion}
Observe by Theorems \ref{hypercomplexthm} and \ref{symplecticthem2} that similarly to the $\SO^*(2n)\Sp(1)$-case, the submodule $\mc{X}_2=[\Lambda^3\E S^3\Hh]^*$ appears in both the hypercomplex and the symplectic torsion. As before we call $\mc{X}_2$ the module corresponding to \textsf{intersection torsion}. Moreover, and since we have exhausted the hypercomplex and symplectic torsion, we can now consider the complement of their union in the intrinsic torsion. This submodule will be called the module corresponding to \textsf{compatibility torsion}, and in this case it consists of the mixed type module $\mc{X}_{57}=\mc{X}_5\oplus\mc{X}_7=[S^3_0\E \Hh]^*\oplus [\E\Hh]^*$. 
Consequently by Theorems \ref{hypercomplexthm} and \ref{symplecticthem2} we obtain
\begin{theorem}\label{compator}
	Let $(H,\omega)$ be an almost hs-H structure with canonical connection $\nabla^{H,\omega}$. Then $H$ is hypercomplex and $\omega$ symplectic, if and only if $(H, \omega)$ is of mixed type $\mc{X}_{57}$, that is the torsion of $\nabla^{H,\omega}$ is contained in the submodule $[S^3_0\E \Hh]^*\oplus[\E\Hh]^*$.
\end{theorem}

\begin{rem}
	\textnormal{Due to \cite[\genCar\ and \zeroprog]{CGWPartI}, the module $\mc{X}_{57}$ appearing in the last conclusion is uniquely fixed by our normalization condition for $\nabla^{H, \omega}$.}
\end{rem}

\section{Integrability conditions via distinguished connections}\label{sec2II}

\subsection{General theory}
We recall that there are several distinguished linear connections that can be used to study $\SO^*(2n)$-structures $(H, \omega)$, or $\SO^*(2n)\Sp(1)$-structures $(Q, \omega)$. For instance, in the first part we used the Obata connection $\nabla^{H}$ and the unimodular Oproiu connection $\nabla^{Q, \vol}$. Obviously there are many other connections which can facilitate an examination of the underlying geometries, such as almost symplectic connections $\nabla^{\omega}$, etc. These connections naturally act on the tensors associated to such structures (see Section 2 of \cite{CGWPartI}), but in general they do not preserve them. For example, for a generic almost hs-H structure $(H, \omega)$ we have $\nabla^{H}\omega\neq 0$, $\nabla^{H}h\neq 0$, $\nabla^{H}\Phi\neq 0$, etc. Nevertheless, we can relate the values of the corresponding covariant derivatives with certain intrinsic torsion components, and hence with 1st-order integrability conditions. Before we proceed with details, let us pose a general result which can be used as a guideline for the discussion that follows. We begin with the following definition.

\begin{defi}
Let $K,L\subset\Gl(n, \R)$ be Lie groups. We will say that a $K$-structure $ \mc{P}_K$ on a smooth manifold $M$ is \textsf{compatible} with a $L$-structure $\mc{P}_L$ on $M$, if $ \mc{P}:= \mc{P}_K\cap \mc{P}_L$ is a $G:=K\cap L$-structure on $M$. This means that for each $x\in M$, there exists a frame $u\in \mc{F}_x(M)$ that is adapted for both the $K$- and the $L$-structure.
	Here, as before $\mc{F}(M)$ denotes the frame bundle of $M$.
\end{defi}

\begin{prop}\label{prophol}
	Let $\mc{P}_K$ be a $K$-structure on an $n$-dimensional manifold $M$ with a $K$-connection $\nabla^K$. 
	Fix $f_0\in V$ for a $\Gl(n,\R)$-module $V$ with stabilizer 
	\[
	L=\{l\in \Gl(n,\R): l\cdot f_0=f_0\}
	\]
	 such that there exists a $G$-invariant direct sum decomposition $\fr{k}=\fr{g}\oplus\fr{m}$, where $G:=K\cap L$.
	Let ${\sf F }$ be a tensor field represented by the $\Gl(n,\R)$-equivariant map $F:\mc{F}(M) \to V$, such that $\mc{P}_L= F^{-1}(f_0)$ is an $L$-structure on $M$ and $\mc{P}_K$ and $\mc{P}_L$ are compatible.
	Then, the following holds for the corresponding $G$-structure $\mc{P}:= \mc{P}_K\cap \mc{P}_L$:\\
\textsf{1)} If $\nabla^L$ is a $L$-connection, then the $(1, 2)$-tensor $\mathscr{A}:=\nabla^L-\nabla^K$ takes values in $T^*M\otimes (\fr{l}+\fr{k})_{\mc{P}}$, where 
\[
(\fr{l}+\fr{k})_{\mc{P}}:= \mc{P}\times_{G}(\fr{l}+\fr{k})\,,
\]
and $\fr{l}, \fr{k}$ are the Lie algebras of $L, K$, respectively.\\
\textsf{2)} The covariant derivative $\nabla^K \mathsf{F}$ is a smooth section of a bundle $\mathscr{E}\to M$ isomorphic to the associated bundle $T^*M\otimes\big(\fr{k}/\fr{g}\big)_{\mc{P}}$, and $\fr{k}/\fr{g}$ is identified with a subspace of $V$ via the action $A \mapsto A \cdot f_0$.\\
\textsf{3)} Let ${\sf s} : \mathscr{E}\to T^*M\otimes\fr{m}_{\mc{P}}$ be a splitting of the natural projection ${\sf p} : T^*M\otimes\fr{k}_{\mc{P}}\to\mathscr{E}$, that is, 
\[
{\sf p}\circ{\sf s}=\id_{\mathscr{E}}\,.
\]
Then, the $K$-connection 
\begin{equation}\label{ncheck}
\nabla:=\nabla^K-\sf{s}(\nabla^K{\sf F})
\end{equation}
preserves ${\sf F}$, i.e., $\nabla{\sf F}=0$. So in particular, $\nabla$ is a $G$-connection.\\
\textsf{4)} Both the intrinsic torsion of the $L$-structure $\mc{P}_L$ and of the $G$-structure $\mc{P}$ are given by appropriate projections of the torsion 
\[
T=T^K-\delta\big({\sf s}(\nabla^K{\sf F})\big)
\]
of $\nabla$ to the corresponding intrinsic torsion module. Here, $\delta : T^*M\otimes \fr{k}_{\mc{P}} \to \Tor(M)$ is the Spencer alternation operator related to the $K$-structure and $T^K$ is the torsion of $\nabla^K$.
\end{prop}
\begin{proof}
Let us choose a (local) frame $u\in\Gamma(\mc{P})$. 
At $u$ the connection 1-form of $\nabla^L$ takes values in $\fr{l}$, and the connection 1-form of $\nabla^K$ in $\fr{k}$. Hence, it follows that their difference tensor $\mathscr{A}$ is a smooth section of $T^*M\otimes (\fr{l}+\fr{k})_{\mc{P}}$. This proves the first claim. As a consequence, and since the tensor ${\sf F}$ has constant coordinates $f_0$ in the frame $u$, the covariant derivative $\nabla^K{\sf F}$ is 
\[
\nabla^K f_0 =\nabla^L f_0-\mathscr{A}\cdot f_0=-\mathscr{A}\cdot f_0\,,
\]
where $\mathscr{A}\cdot f_0$ encodes the natural action of 1-forms with values in $\Ed(TM)$ on the tensor bundle over $M$. The latter action is linear, with kernel given by 1-forms with values in the Lie algebra $\fr{l}$. Therefore its image $\mathscr{A}\cdot{\sf F}$ is a section of a bundle $\mathscr{E}$ over $M$, which corresponds to the section ${\sf p}(\mathscr{A})$ of the following associated bundle
\[
T^*M\otimes \big((\fr{k}+\fr{l})/\fr{l}\big)_{\mc{P}} = T^*M\otimes\big(\fr{k}/\fr{g}\big)_{\mc{P}}\,.
\]
This provides the isomorphism 
\[
\mathscr{A}\cdot{{\sf F}}\in\Gamma(\mathscr{E})\longmapsto {\sf p}(\mathscr{A})\in\Gamma(T^*M\otimes\big(\fr{k}/\fr{g}\big)_{\mc{P}})\,.
\]
Next, under our assumptions and by the definition of $\nabla$ we have
\[
\nabla{\sf F}:=\nabla^K{\sf F}-\big({\sf s}(\nabla^K{\sf F})\big)\cdot{\sf F}=\nabla^K{\sf F}-{\sf p}\big({\sf s}(\nabla^K{\sf F})\big)=\nabla^K{\sf F}-\nabla^K{\sf F}=0\,.
\]
This yields the third assertion, which moreover shows that the connection $\nabla$ is both an $L$-connection and a $G$-connection. Thus the intrinsic torsion of the $L$-structure and the intrinsic torsion of the $G$-structure are both given by the projection of $T$ to the corresponding intrinsic torsion module.
This proves the final assertion.
The stated formula for the expression of $T$ follows easily by (\ref{ncheck}).
\end{proof}
We may apply Proposition \ref{prophol} to $\SO^*(2n)$- and $\SO^*(2n)\Sp(1)$-structures.
In Table \ref{tabKL} we summarize the Lie groups $K,L$ provided by the defining tensors ${\sf F}$ of $\SO^*(2n)$- and $\SO^*(2n)\Sp(1)$-structures.
However, let us postpone that, and initially illustrate Proposition \ref{prophol} by a more characteristic example, related to metric connections.
\begin{example}
Set $K:=\Gl(n, \R)$, $L:=\Oo(n)$, and let $g$ be a Riemannian metric on a $n$-dimensional manifold $M$. Then $G=K\cap L=\Oo(n)=L$ and we denote by $\mc{O}(M)\to M$ the orthonormal frame bundle. A reductive complement of $\fr{o}(n)$ in $\fr{k}$ coincides with the space of symmetric 2-tensors on $\R^n$, that is 
\[
\fr{gl}(n, \R)=\fr{o}(n)\oplus\fr{m}\cong\Lambda^{2}(\R^{n})^*\oplus S^{2}(\R^n)^*\,,\qquad \fr{m}\cong S^{2}(\R^n)^*\,.
\]
Let $\nabla^{K}$ be a torsion-free linear connection on $M$. Then, the covariant derivative $\nabla^{K}g$ takes values in $T^*M\otimes\fr{m}_{\mc{O}(M)}$. For some (1, 2)-tensor field $A$ on $M$, it is easy to see that the linear connection $\nabla=\nabla^{K}+A$ is a metric connection if and only if
\[
(\nabla^{K}_{X}g)(Y, Z)=g(A_{X}Y, Z)+g(Y, A_{X}Z)=A(X, Y, Z)+A(X, Z, Y)
\]
where we set $A(X, Y, Z)=g(A_{X}Y, Z)$ and $A_{X}Y=A(X, Y)$. Since the quantity in the left hand side should be symmetric in the last two indices, the above relation is equivalent to say that
\begin{equation}\label{metricA}
g(A(X, Y), Z)=\frac{1}{2}(\nabla^{K}_Xg)(Y, Z)\,,\quad \forall \ X, Y, Z\in\Gamma(TM)\,.
\end{equation}
In other words, for the splitting map ${\sf s} : \Gamma(\mathscr{E})\to \Gamma(T^*M\otimes\fr{m}_{\mc{O}(M)})$ we get that the image $-{\sf s}(\nabla^{K}g)$ is the (1, 2)-tensor field $A$ defined by (\ref{metricA}). 
Therefore
\[
	\nabla:=\nabla^{K}-{\sf s}(\nabla^{K}g)=\nabla^{K}+A
\]
is a metric connection with respect to $g$, that is $\nabla g=0$. If we want to obtain the Levi-Civita connection, i.e., $\nabla=\nabla^{g}$, we should choose a different complement $S^2T^*M\otimes TM$. This is because in this case $\nabla$ is also a torsion-free connection, so the difference $\nabla-\nabla^{K}$ should take values in $S^2T^*M\otimes TM$. In addition, the related isomorphism ${\sf s} : \mathscr{E}\to S^2T^*M\otimes TM$ is encoded by the well-known Koszul formula.
\end{example}

\begin{table}[h]
\begin{tabular}{l|c|c|c|c|c}
${\sf F}$& $\omega$& $\omega$& $H$& $h$& $\Phi$\\
\hline
$K$& $\Gl(n,\Hn)$ & $\Gl(n,\Hn)\Sp(1)$&$\Sp(4n,\R)$& $\Gl(4n,\R)$ & $\Gl(4n,\R)$\\
\hline
$L$& $\Sp(4n,\R)$& $\Sp(4n,\R)$& $\Gl(n,\Hn)$& $\SO^*(2n)\Sp(1)$& $\SO^*(2n)\Sp(1)$ \\
\hline
$G$& $\SO^*(2n)$&$\SO^*(2n)\Sp(1)$& $\SO^*(2n)$& $\SO^*(2n)\Sp(1)$& $\SO^*(2n)\Sp(1)$ \\
\end{tabular}
\vspace{0.5cm}
\caption{\small Examples of groups $K, L, G=K\cap L$ and the corresponding tensor ${\sf F}$}\label{tabKL}
\end{table}
\begin{rem}
\textnormal{We proceed with a detailed investigation of most of the cases from Table \ref{tabKL}. Let us emphasize on the fact that the connection $\nabla$ of Proposition \ref{prophol} does {\it not} have to be a minimal connection. However, for particular cases (as for example the connection $\nabla^{H, \omega}$ introduced in \cite[\starconnection]{CGWPartI}), one can derive minimality with respect to certain normalization conditions, as we do in \cite[\zeroprog]{CGWPartI}.
On the other hand, in general the explicit descriptions of the isomorphism $\mathscr{E}\cong T^*M\otimes\big(\fr{k}/\fr{g}\big)_{\mc{P}}$, and also of the splitting map ${\sf s}$ appearing in Proposition \ref{prophol}, are both highly non-trivial tasks.}
\end{rem}

\subsection{The contribution of the Obata connection}

\smallskip
Let us consider the following situation:
\[
	K:=\Gl(n, \Hn)\,,\quad L:=\Sp(4n, \R)\,,\quad {{\sf F}}:=\omega, \quad \nabla^K=\nabla^{H}\,, 
\]
where $H$ is an almost hypercomplex structure on a $4n$-dimensional manifold $M$ and $\omega$ is a scalar 2-form on $M$ with respect to $H$. This means that $G=K\cap L=\SO^*(2n)$ and we write $\pi : \mc{P}\to M$ for the corresponding $\SO^*(2n)$-structure on $M$. In this case we have $\fr{m}\cong[S^2\E]^*$ and $\nabla^{H}\omega$ is a smooth section of the bundle
\[
	\mathscr{E}\cong T^*M\otimes \fr{m}_{\mc{P}}=T^*M\otimes ([S^2\E]^*)_{\mc{P}}\,.
\]
For the splitting map ${\sf s} : \Gamma(\mathscr{E})\to \Gamma(T^*M\otimes\fr{m}_{\mc{P}})$ we get that the image $-{\sf s}(\nabla^{H}\omega)$ is the (1, 2)-tensor field $A$ introduced in \cite[\starconnection]{CGWPartI}, satisfying
\begin{equation}\label{omegA}
\omega(A(X, Y), Z)=\frac{1}{2}(\nabla^{H}_X\omega)(Y, Z)\,,\quad \forall \ X, Y, Z\in\Gamma(TM)\,.
\end{equation}
Consequently, we deduce that the connection
\[
	\nabla=\nabla^{H}-{\sf s}(\nabla^{H}\omega)=\nabla^{H}+A=:\nabla^{H, \omega}
\]
is a $\SO^*(2n)$-connection, and hence $\nabla^{H, \omega}{{\sf F}}=\nabla^{H, \omega}\omega=0$. This provides an alternative proof of \cite[\starconnection]{CGWPartI} for the fact that $\nabla^{H, \omega}$ is an adapted connection. 

Note that when the scalar 2-form $\omega$ is $\nabla^{H}$-parallel, i.e.,
\[
	(\nabla^{H}_X\omega)(Y, Z)=0\,,\quad \forall \ X, Y, Z\in\Gamma(TM)\,,
\]
then by the non-degeneracy of $\omega$ we obtain the vanishing of $A$, and hence $\nabla^{H, \omega}=\nabla^{H}$. We see that
\begin{lem}\label{nocodaz}
The scalar 2-form $\omega$ is $\nabla^{H}$-parallel if and only if 
\begin{equation}\label{complic1}
	(\nabla_{X}^{H}\omega)(Y, Z)=(\nabla_{Y}^{H}\omega)(X, Z)\,,
\end{equation}
for any $X, Y, Z\in\Gamma(TM)$. 
\end{lem}
\begin{proof}
We prove only the converse direction, which is less trivial. Assume that (\ref{complic1}) holds. This means that
the tensor field $A$ takes values in the first prolongation $\fr{gl}(n, \Hn)^{(1)}$, which is trivial. So the claims follows.
\end{proof}

Let us now examine 1st-order integrability conditions in terms of the Obata connection $\nabla^{H}$. 
\begin{prop}\label{THo}
The torsion of $\nabla^{H, \omega}$ is given by
\[
T^{H, \omega}(X, Y)=T^{H}(X, Y)+\delta(A)(X, Y)\,, \quad\forall \ X, Y\in\Gamma(TM)\,.
\]
Thus, 
the adapted connection $\nabla^{H, \omega}$ is torsion-free if and only if $T^{H}$ satisfies 
\begin{equation}\label{omegTH}
	T^{H}=0\,,\quad \text{and}\quad \nabla^{H}\omega=0\,.
\end{equation}
In other words, $T^{H, \omega}=0$ if and only if $H$ is 1-integrable and $\nabla^{H, \omega}=\nabla^{H}$.
\end{prop}
\begin{proof}
The first statement is easy and implies that $T^{H, \omega}=0$ if and only if 
\[
	\omega(T^{H}(X, Y), Z)+\frac{1}{2}\Big((\nabla_{X}^{H}\omega)(Y, Z)-(\nabla_{Y}^{H}\omega)(X, Z)\Big)=0\,,
\]
for any $X, Y, Z\in\Gamma(TM)$. By Theorem \ref{hypercomplexthm} we have that $ T^{H}\in\mc{X}_{126}$ and by \cite[\zeroprog]{CGWPartI} we conclude that $\delta(A)\in\mc{X}_{3457}$. 
Since they belong to different components, they should vanish simultaneously and in combination with Lemma \ref{nocodaz} we derive the stated conditions in (\ref{omegTH}). In particular, note that $T^{H, \omega}=0$, if and only if 
$\delta(A)=0$ identically, i.e., $A$ is symmetric and $T^{H}$ vanishes. By Lemma \ref{nocodaz} the first condition is equivalent to say that $A=0$, i.e., $\nabla^{H, \omega}=\nabla^{H}$.
\end{proof}
Let us now examine the closedness of the scalar 2-form $\omega$. 
\begin{prop}\label{dwH}
Let $(M, H=\{J_{a} : a=1, 2, 3\}, \omega)$ be an almost hs-H manifold. Then, the differential of $\omega$ satisfies 
\begin{equation}\label{domega}
\dd\omega(X, Y, Z)=\pi_{\omega}(T^{H, \omega})(X, Y, Z):=\fr{S}_{X, Y, Z}\omega(T^{H, \omega}(X, Y), Z)\,, 
\end{equation}
for any $X, Y, Z\in\Gamma(TM)$, and hence if $T^{H, \omega}=0$ then $\dd\omega=0$. More general, $\dd\omega=0$ or equivalently $(M, H, \omega)$ is of type $\mc{X}_{1567}$, if and only if 
\[
\pi_{\omega}(T^{H})=0\,,\quad\text{and}\quad \fr{S}_{X, Y, Z}(\nabla^{H}_{X}\omega)(Y, Z)=0\,,
\]
for any $X, Y, Z\in\Gamma(TM)$. 
\end{prop}
\begin{proof}
The linear connection $\nabla^{H, \omega}$ has torsion given by $T^{H, \omega}$ by definition, and consequently the differential of $\omega$ is given by 
\begin{eqnarray*}
\dd\omega(X, Y, Z)&=&(\nabla^{H, \omega}_{X}\omega)(Y, Z)-(\nabla^{H, \omega}_{Y}\omega)(X, Z)+(\nabla^{H, \omega}_{Z}\omega)(X, Y)\\
&&+\omega(T^{H, \omega}(X, Y), Z)-\omega(T^{H, \omega}(X, Z), Y)+\omega(T^{H, \omega}(Y, Z), X) \\
&=&\fr{S}_{X, Y, Z}\omega(T^{H, \omega}(X, Y), Z)\,,
\end{eqnarray*}
where the second equality follows since $\nabla^{H, \omega}$ preserves the pair $(H, \omega)$ (and so the first line must vanish). Finally, regarding the closedness of $\omega$, note that the second condition $\fr{S}_{X, Y, Z}(\nabla^{H}_{X}\omega)(Y, Z)=0$ is equivalent to $\fr{S}_{X, Y, Z}\omega(\delta(A)(X, Y), Z)=0$, for any $X, Y, Z\in\Gamma(TM)$. 
\end{proof}
We summarize our conclusions about the scalar 2-form $\omega$ when the torsion of $\nabla^{H, \omega}$ vanishes. The almost symplectic structure induced by the scalar 2-form $\omega$ is closed. Hence, $(\omega, \nabla^{H, \omega})$ is a Fedosov structure in the sense that $\nabla^{H, \omega}$ is a compatible torsion-free connection of the symplectic form $\omega$. However, we should mention that the Darboux coordinate frames do not sit in the corresponding reduction to $\SO^*(2n)$, so there is a smooth map
\[
	f : M\to \Sp(4n,\R)/\SO^*(2n)
\]
describing this difference. Such an obstruction can be viewed as a smooth section of the quotient bundle $\mc{S}(M)/\SO^*(2n)\cong \mc{S}(M)\times_{\Sp(4n, \R)}(\Sp(4n,\R)/\SO^*(2n))$. 
Here, $\mc{S}(M)$ denotes the symplectic frame bundle, i.e., the principal $\Sp(4n, \R)$-bundle over $M$ consisting of all symplectic bases with respect to $\omega$.

\subsection{The contribution of Oproiu connections}

\smallskip
As a second main application of Proposition \ref{prophol} we consider the following situation:
\[
	K:=\Gl(n, \Hn)\Sp(1)\,,\quad L:=\Sp(4n, \R)\,,\quad {{\sf F}}:=\omega, \quad \nabla^K=\nabla^{Q}\,, 
\]
where $Q$ is an almost quaternionic structure on a $4n$-dimensional manifold $M$ and $\omega$ is a scalar 2-form on $M$ with respect to $Q$. We have $G=K\cap L=\SO^*(2n)\Sp(1)$ and denote by $\pi : \mc{Q}\to M$ the corresponding $\SO^*(2n)\Sp(1)$-structure on $M$. Note that an Oproiu connection $\nabla^{Q}$ plays the role of the $K$-connection. 
By \cite{CGWPartI} we know that the covariant derivative $\nabla^Q \omega$ has values in
\[
[\E\Hh]^*\otimes (\langle \omega_0\rangle \oplus [S^2_0\E]^*)\cong [\E\Hh]^*\otimes (\R \cdot\Id \ \oplus \frac{\fr{sl}(n,\Hn)}{\so^\ast(2n)})\,.
\]
This corresponds to $\fr{m}=\R \cdot\Id \ \oplus \frac{\fr{sl}(n,\Hn)}{\so^\ast(2n)}$ and 
\[
	\mathscr{E}\cong T^*M\otimes \fr{m}_{\mc{Q}}\,.
\]
However, in Theorem \cite[\starconnectionsp]{CGWPartI}, for the map ${\sf p} : T^*M\otimes\cc\fr{gl}(n, \Hn)\oplus\fr{sp}(1)\rr_{\mc{Q}}\to\mathscr{E}$ we have chosen the following splitting
\[
	{\sf s} : \mathscr{E}\to\cc\ke \delta\oplus ([\E\Hh]^*\otimes [S^2_0\E]^*)\rr_{\mc{Q}}\subset T^*M\otimes\cc\fr{gl}(n, \Hn)\oplus\fr{sp}(1)\rr_{\mc{Q}}\,.
\]
This is the complement of $T^*M\otimes\cc\fr{so}^*(2n)\oplus\fr{sp}(1)\rr_{\mc{Q}}$ in $T^*M\otimes\cc\fr{gl}(n, \Hn)\oplus\fr{sp}(1)\rr_{\mc{Q}}$, and one should observe that this is a different complement than 
\[
	T^*M\otimes\fr{m}_{\mc{Q}}\,.
\]
In particular, by the proof of \cite[\starconnectionsp]{CGWPartI} we deduce that ${\sf s}={\sf s}_1+{\sf s}_2$, where ${\sf s}_1$ takes values in $\ke \delta$ and ${\sf s}_2$ takes values in $[\E\Hh]^*\otimes [S^2_0\E]^*$, respectively. These are respectively given by
\begin{eqnarray*}
{\sf s}_1(\nabla^Q \omega)&=&-\frac{\Tr_2(A)}{4(n+1)} \otimes \id+\pi_A(\omega\otimes \frac{\Tr_2(A)^T}{(n+1)})- \pi_S(\omega\otimes \frac{\Tr_2(A)^T}{(n+1)})+\sum_{a=1}^3 \frac{\Tr_2(A)}{4(n+1)}\circ J_a\otimes J_a\\
{\sf s}_2(\nabla^Q \omega)&=&-A+\frac{\Tr_2(A)}{4(n+1)} \otimes \id-\pi_A(\omega\otimes \frac{\Tr_2(A)^T}{(n+1)})\,,
\end{eqnarray*}
where $A$ is a (1, 2)-tensor field satisfying
$
\omega(A(X, Y), Z)=\frac{1}{2}(\nabla^{Q}_X\omega)(Y, Z)
$, for any $X,Y,Z\in \Gamma(TM)$, $\Tr_{2}(A)(X):=\Tr(A(X, \cdot))$, and $\pi_{A}$ is the projection defined by 
\[
\pi_A(\omega\otimes Z)(X,Y):=Asym\Big(\pi_{1, 1}(\omega(X,.) \otimes Z)\Big)Y\,,
\]
see \cite[Section 3.2]{CGWPartI} for details. The reason for this decomposition is that 
\[
\nabla^{Q,\vol}=\nabla^{Q}-{\sf s}_1(\nabla^Q \omega)
\]
is the unimodular Oproiu connection with respect to $\vol=\omega^{2n}$. Moreover, we see that the tensor field $A^{\vol}$ defined by the relation 
\[
\omega\cc A^{\vol}(X, Y), Z\rr=\frac{1}{2}(\nabla^{Q,\vol}_{X}\omega)(Y, Z)\,,\quad \forall \ X, Y, Z\in\Gamma(TM)
\]
satisfies
\[
A^{\vol}=-{\sf s}_2(\nabla^Q \omega) \quad {\rm and} \quad \Tr_2(A^{\vol})=0\,.
\]
As a conclusion, we arrive to the formula 
\[
\nabla^{Q,\omega}=\nabla^Q-{\sf s}(\nabla^Q \omega)=\nabla^{Q}-{\sf s}_1(\nabla^Q \omega)-{\sf s}_2(\nabla^Q \omega)=\nabla^{Q,\vol}+A^{\vol}
\]
for the $\SO^*(2n)\Sp(1)$-connection $\nabla^{Q, \omega}$ from \cite[\starconnectionsp]{CGWPartI}. Let us now pose the following 

\begin{lem}\label{deltaAvol}
The condition $\delta(A^{\vol})=0$ is equivalent to 
\begin{eqnarray*}
(\nabla^{Q}_X\omega)(Y,Z)&=&\frac{\Tr_2(A)(X)}{2(n+1)} \omega(Y,Z)-\omega(\pi_A(\omega(X,.)\otimes \frac{2\Tr_2(A)^T}{(n+1)})Y,Z)\quad\quad ({\sf i}) \\
&=&\frac{1}{4(n+1)}\Big(\Tr_2(A)(X)\omega( Y,Z)+\sum_a \Tr_2(A)(J_aX) g_{J_a}(Y,Z)-\omega(X,Y) \Tr_2(A)(Z)\\
&&-\sum_a g_{J_a}(X,Y)\Tr_2(A)(J_a Z)\Big)\quad\quad\quad\quad\quad\quad\quad\quad\quad\quad\quad\quad({\sf ii})\\
&=&\frac{\Tr_2(A)\cc h(Y,Z)X-h(X,Y)Z\rr}{4(n+1)}\quad\quad \quad\quad\quad\quad\quad\quad\quad\quad\quad\ ({\sf iii})
\end{eqnarray*}
for all $X,Y,Z\in \Gamma(TM)$, where the tensor field $A$ is given by 
\begin{equation}\label{wA}
\omega(A(X, Y), Z)=\frac{1}{2}(\nabla^{Q}_X\omega)(Y, Z)\,, \forall \ X,Y,Z\in \Gamma(TM)\,.
\end{equation}
In particular,
\begin{equation}\label{Avol}
\omega\cc A^{\vol}(X, Y), Z\rr=\frac12(\nabla^{Q}_X\omega)(Y,Z)-\frac{\Tr_2(A)\cc h(Y,Z)X-h(X,Y)Z\rr}{8(n+1)}
\end{equation}
for any $X,Y,Z\in \Gamma(TM).$
\end{lem}
\begin{proof}
Since $A^{\vol}$ belongs to a complement of $\ke(\delta)$, by \cite[\zeroprog]{CGWPartI} we deduce that $\delta(A^{\vol})=0$, if and only if ${\sf s}_2(\nabla^Q \omega)=0$. Thus, the first relation $({\sf i})$ follows by the definition given above for ${\sf s}_2$ and $A$. By \cite[Section 3.2]{CGWPartI} we also deduce that 
\begin{eqnarray*}
\omega(\pi_A(\omega(X,\cdot)\otimes \frac{2\Tr_2(A)^T}{(n+1)})Y,Z)&=&\omega\CC\frac{1}{8}\big(\omega(X,Y) \frac{2\Tr_2(A)^T}{(n+1)}
-\omega(X,\frac{2\Tr_2(A)^T}{(n+1)}) Y\\
&&-\sum_a g_{J_a}(X,Y) J_a\frac{2\Tr_2(A)^T}{(n+1)}+\sum_a g_{J_a}(X,\frac{2\Tr_2(A)^T}{(n+1)}) J_aY\big),Z\RR\\
&=&\frac{1}{4(n+1)}\Big(\omega(X,Y) \omega(\Tr_2(A)^T,Z)-\omega(X,\Tr_2(A)^T) \omega(Y,Z)\\
&&-\sum_a g_{J_a}(X,Y) \omega(J_aTr_2(A)^T,Z)+\sum_a g_{J_a}(X,\Tr_2(A)^T) \omega(J_aY,Z)\Big)\\
&=&\frac{1}{4(n+1)}\Big(\omega(X,Y) \Tr_2(A)(Z)+Tr_2(A)(X) \omega(Y,Z)\\
&&+\sum_a g_{J_a}(X,Y)\Tr_2(A)(J_aZ)-\sum_a \Tr_2(A)(J_aX) g_{J_a}(Y,Z)\Big)\,.
\end{eqnarray*}
This formula is combined with the first term $\frac{\Tr_2(A)(X)}{2(n+1)} \omega(Y,Z)$ and gives the second stated formula $({\sf ii})$. Finally, the last equality ({\sf iii}) follows by the definition of the quaternionic skew-Hermitian form $h$, while the relation (\ref{Avol}) is also a simple consequence of the above.
\end{proof}

Due to this description we are now able to proceed by specifying 1st-order integrability conditions for $\SO^*(2n)\Sp(1)$-structures in terms of the connections $\nabla^{Q}$ and $\nabla^{Q, \vol}$.
\begin{prop}
The torsion of $\nabla^{Q, \omega}$ is given by
\[
T^{Q, \omega}(X, Y)=T^{Q}(X, Y)+\delta(A^{\vol})(X, Y)\,, \quad\forall \ X, Y\in\Gamma(TM)\,,
\]
where
\begin{eqnarray*}
\omega(\delta(A^{\vol})(X,Y),Z)&=&\frac12(\nabla^{Q}_X\omega)(Y,Z)-\frac12(\nabla^{Q}_Y\omega)(X,Z)\\
&&+\frac{\Tr_2(A)\cc h(X,Z)Y-h(Y,Z)X+2\omega(X,Y)Z\rr}{8(n+1)}
\end{eqnarray*}
with $A$ defined by {\rm (\ref{wA})}. 
Thus, the adapted connection $\nabla^{Q, \omega}$ is torsion-free if and only if
\begin{equation}\label{omegTHH}
T^{Q}=0\,,\quad \text{and}\quad (\nabla^{Q}_X\omega)(Y,Z)=\frac{\Tr_2(A)\cc h(Y,Z)X-h(X,Y)Z\rr}{4(n+1)}\,.
\end{equation}
Thus, $\nabla^{Q, \omega}=\nabla^{Q, \vol}+A^{\vol}$ is torsion-free if and only if $\nabla^{Q, \vol}\omega=0$ and $T^{Q, \vol}=0$. 
In other words, $T^{Q, \omega}=0$ if and only if $\nabla^{Q, \omega}=\nabla^{Q, \vol}$, and $(Q, \vol)$ is 1-integrable.
\end{prop}
\begin{proof}
Since all of the Oproiu connections have the same torsion, i.e., $T^Q=T^{Q, \vol}$ (see e.g. \cite{AM}), it remains for us to compute $\delta(A^{\vol})$.
In particular, the first relation is an immediate consequence of the relation $T^Q=T^{Q, \vol}$ and the definition of $\nabla^{Q, \omega}$. 
For the explicit formula of $\delta(A^{\vol})$ we apply Lemma \ref{deltaAvol}, which means that we obtain the stated formula by skew-symmetrization of (\ref{Avol}) in the first two indices, and using the definition of $h$.
In addition, the stated integrability conditions for the vanishing of $T^{Q, \omega}$ also occur as a direct consequence of Lemma \ref{deltaAvol}. This is because $T^{Q}\in\mc{X}_{12}$ and $\delta(A^{\vol})\in\mc{X}_{345}$, so these components vanish simultaneously. The last claim follows by the vanishing of $\Tr_{2}(A^{\vol})$. 
\end{proof}

As for the differential of the scalar 2-form $\omega$, we get the following. 
\begin{prop}
Let $(M, Q, \omega)$ be an almost qs-H manifold. Then, the differential of the scalar 2-form $\omega$ satisfies 
\begin{equation}\label{domega}
\dd\omega(X, Y, Z)=\pi_{\omega}(T^{Q, \omega})(X, Y, Z)=\fr{S}_{X, Y, Z}\omega(T^{Q, \omega}(X, Y), Z)\,, 
\end{equation}
for any $X, Y, Z\in\Gamma(TM)$, and hence if $T^{Q, \omega}=0$ then $\dd\omega=0$. More generally, $\dd\omega=0$ or equivalently $(M, Q, \omega)$ is of type $\mc{X}_{15}$, if and only if 
\[
\pi_{\omega}(T^{Q})=0\,,\quad\text{and}\quad \fr{S}_{X, Y, Z}(\nabla^{Q}_{X}\omega)(Y, Z)= \fr{S}_{X, Y, Z}\frac{\Tr_2(A)\cc h(Y,Z)X-h(X,Y)Z\rr}{4(n+1)}\,,
\]
for any $X, Y, Z\in\Gamma(TM)$. 
\end{prop}

\subsection{The use of the fundamental symmetric 4-tensor} 
Let us now discuss the case
\[
	K:=\Gl(4n, \R)\,,\quad L:=\SO^*(2n)\Sp(1)\,,\quad {{\sf F}}:=\Phi, \quad \nabla^K=\nabla\,, 
\]
where $\nabla$ is any affine connection on $(M, Q,\omega)$. We have $G=K\cap L=L=\SO^*(2n)\Sp(1)$ and $\Phi\in\Gamma(S^{4}T^*M)$ is the fundamental 4-tensor of the corresponding $G$-structure, which we will denote again by $\pi : \mc{Q}\to M$. We will work in an algebraic setting by using the 4-tensor $\Phi_0$, which has stabilizer $L=G$, see \cite{CGWPartI}.
The Lie algebra $\fr{k}=\fr{gl}(4n, \R)$ of $K$ will be identified with the $G$-module of endomorphisms of $[\E\Hh]$, i.e.,
\[
\fr{k}=\fr{gl}(4n, \R)\cong \Ed([\E\Hh])\,.
\]
Recall that $\Ed([\E\Hh])$ admits the following $\SO^*(2n)\Sp(1)$-equivariant decomposition (see \cite[Section 1.3]{CGWPartI}) 
	\[
	\Ed([\E\Hh]) \cong \R\cdot\Id \oplus \ \fr{sp}(1) \oplus \so^\ast(2n) \oplus \frac{\fr{sl}(n,\Hn)}{\so^\ast(2n)} \oplus\frac{\sp(\omega_0)}{\so^\ast(2n)\oplus \sp(1)} \oplus [\Lambda^2\E S^2\Hh]^*\,.
\]
Here $\R\cdot\id\cong \langle\omega_0\rangle$ and
\[
\fr{sp}(1)\cong [S^2\Hh]^*\,,\quad \fr{so}^*(2n)\cong [\Lambda^{2}\E]^*\,,\quad \displaystyle\frac{\fr{sl}(n,\Hn)}{\so^\ast(2n)}\cong [S^2_0\E]^*\,,\quad 
\displaystyle\frac{\sp(\omega_0)}{\so^\ast(2n)\oplus \sp(1)}\cong [S^2_0\E S^2\Hh]^*\,.
\]
By the above decomposition, we deduce that the splitting 
\[
\fr{k}=\fr{gl}(4n, \R)=\cc\fr{so}^*(2n)\oplus\fr{sp}(1)\rr\oplus\fr{m}=\fr{g}\oplus\fr{m}\,,
\]
defines a reductive decomposition of $\fr{k}$, where the reductive complement $\fr{m}$ is given by
\[
\fr{m}:=\CC\R\cdot\Id \ \oplus \frac{\fr{sl}(n,\Hn)}{\so^\ast(2n)} \oplus\frac{\sp(\omega_0)}{\so^\ast(2n)\oplus \sp(1)} \oplus [\Lambda^2\E S^2\Hh]^*\RR\,.
\]
Thus, according to Proposition \ref{prophol} the covariant derivative $\nabla\Phi$ takes values in the associated bundle $\mathscr{E}$ with fiber isomorphic to
\[
[\E\Hh]^*\otimes \fr{m}=:\mc{W}_1\oplus\mc{W}_2\oplus\mc{W}_3\oplus\mc{W}_4\,,
\]
where we set
\begin{eqnarray*}
\mc{W}_1&:=&[\E\Hh]^*\otimes \Id\,,\\
\mc{W}_2&:=&[\E\Hh]^*\otimes\frac{\fr{sl}(n,\Hn)}{\so^\ast(2n)}\cong [\E\Hh]^*\otimes [S^2_0\E]^*\,,\\
\mc{W}_3&:=&[\E\Hh]^*\otimes\frac{\sp(\omega_0)}{\so^\ast(2n)\oplus \sp(1)}\cong [\E\Hh]^*\otimes [S^2_0\E S^2\Hh]^*\,,\\
\mc{W}_4&:=&[\E\Hh]^*\otimes[\Lambda^2\E S^2\Hh]^*\,,
\end{eqnarray*}
respectively. Note that $\mc{W}_2, \mc{W}_3$ and $\mc{W}_4$ are reducible as $G$-modules. One of our goals below is to provide their expressions into irreducible submodules, and to relate them to the algebraic types $\mc{X}_{i_{1}\ldots i_{j}}$ for $1\leq i_{1}<\ldots< {i_{j}}\leq5$. Moreover, we will explicitly specify the isomorphism
\[
{\sf s}: \mathscr{E}\to (\mc{W}_1\oplus \mc{W}_2\oplus \mc{W}_3 \oplus \mc{W}_4)_{\mc{Q}}\,.
\]
We begin with the following lemma.
\begin{prop}\label{purew}
	The projections ${\sf w}_a : [\E\Hh]^*\otimes \Ed([\E\Hh])\to \mc{W}_a$ $(a=1,2,3,4)$ given below, are equivariant and independent of the choice of an admissible basis $\{J_a\,,\ a=1,2,3\}$ for the standard quaternionic structure $Q_0$ on $[\E\Hh]$:
\begin{enumerate}
\item[${\sf 1)}$] ${\sf w}_1: [\E\Hh]^*\otimes \Ed([\E\Hh])\to \mc{W}_1$ defined by
\[{\sf w}_1(A)(X,Y):=\frac{\Tr(A_X)}{4n}Y.\]
In particular, the elements of $ \mc{W}_1$ take form $\upalpha_1=\xi\otimes \id$ for some $\xi\in [\E\Hh]^*$.
\item[${\sf 2)}$] ${\sf w}_2: [\E\Hh]^*\otimes \Ed([\E\Hh])\to \mc{W}_2$ defined by
\[{\sf w}_2(A)(X,Y):=\pi_{A}(A_X)Y-\frac{\Tr(A_X)}{4n}Y\] 
In particular, the pure elements $\upalpha_2\in \mc{W}_2$ are given by $\omega_{0}(\upalpha_2\cdot, \cdot)=\xi\otimes \hat{\omega}$ for some $\xi\in [\E\Hh]^*$ and a scalar 2-form $\hat{\omega}\in [S^2_0\E]^*$, i.e., $\Tr_2(\upalpha_2)=0$.
\item[${\sf 3)}$] ${\sf w}_3: [\E\Hh]^*\otimes \Ed([\E\Hh])\to \mc{W}_3$ defined by
\[{\sf w}_3(A)(X,Y):=Sym(A_X-\pi_{1,1}(A_X))Y+\sum_a \frac{\Tr(A_X\circ J_a)}{4n}J_aY\] 
In particular, the pure elements $\upalpha_3\in \mc{W}_3$ are given by $\omega_{0}(\upalpha_3\cdot, \cdot)=\xi\otimes \hat{g}_{J_b}$ for some $\xi\in [\E\Hh]^*$, $J_b\in \{J_a\,,\ a=1,2,3\}$ and a scalar 2-form $\hat{\omega}\in [S^2_0\E]^*$, i.e., $\hat{g}_{J_b}(\cdot, \cdot)=\hat{\omega}(\cdot, J_b\cdot)\in [S^2_0\E]^*\otimes\fr{sp}(1)$ and $\Tr_2(\upalpha_3\circ J_b)=0.$
\item[${\sf 4)}$] ${\sf w}_4: [\E\Hh]^*\otimes \Ed([\E\Hh])\to \mc{W}_4$ defined by
\[{\sf w}_4(A)(X,Y):=Asym(A_X-\pi_{1,1}(A_X))Y\] 
In particular, the pure elements $\upalpha_4\in \mc{W}_4$ are given by $\omega_0(\upalpha_4 \cdot, \cdot)=\xi\otimes \rho(\cdot,J_b\cdot )$ for some $\xi\in [\E\Hh]^*$, $J_b\in \{J_a\,,\ a=1,2,3\}$ and quaternionic-Hermitian Euclidean metric $\rho\in [\Lambda^2\E]^*$.
\end{enumerate}
\end{prop}
\begin{proof}
The claimed form for the pure elements $\upalpha_i\in \mc{W}_i$ $(i=1, \ldots, 4)$ follows by Proposition 1.7 and Remark 1.8 in \cite{CGWPartI}. 
By \cite[Section 1.4]{AM}, we have
\[
\imm(\pi_{1,1})=\gl(n,\Hn)\,,\quad \ke(\pi_{1,1})=\frac{\sp(\omega_0)}{\so^\ast(2n)} \oplus [\Lambda^2\E S^2\Hh]^*\cong \sp(1)\oplus [S_0^2\E S^2\Hh]^*\oplus[\Lambda^2\E S^2\Hh]^*\,,
\]
and also the explicit form of the projections to the trace components $\R\cdot \id$ and $\sp(1)$, that is 
\[
\fr{gl}(4n, \R)\ni A\longmapsto \frac{\Tr(A)}{4n}\Id\in\R\Id\,,\quad \fr{gl}(4n, \R)\ni A\longmapsto -\sum_a \frac{\Tr(A\circ J_a)}{4n}J_a\in\fr{sp}(1)\,.
\]
Thus, what remains is the application of the anti/symmetrization operators $Asym$ and $Sym$, which provide the correct projections according to the $\SO^*(2n)\Sp(1)$-decompositions given in \cite[\modules]{CGWPartI}. For instance, consider the case of ${\sf w}_3: [\E\Hh]^*\otimes \Ed([\E\Hh])\to \mc{W}_3=[\E\Hh]^*\otimes [S^2_0\E S^2\Hh]^*$. Let us recall the following $\SO^{\ast}(2n)\Sp(1)$-equivariant decomposition from Part I:
\[
S^2[\E\Hh]^* \cong [\Lambda^2\E]^*\oplus [S^2_0\E S^2\Hh]^*\oplus \sp(1)\,. 
\]
This means that, for an element $A_{X}-\pi_{1, 1}(A_X)$ in the kernel of $\pi_{1, 1}$, the symmetrization must belong to $[S^2_0\E S^2\Hh]^*\oplus \sp(1)$. Thus, by subtracting the trace component $\sp(1)$ we obtain the stated formula for $\w_3$.
The rest of the cases are treated similarly.
\end{proof}
Let us now describe how the pure elements $\upalpha_i\in \mc{W}_i$ from Lemma \ref{purew} act on the fundamental 4-tensor $\Phi_0$.
\begin{lem}\label{pureact}
\noindent $\bullet$ \ For $\upalpha_1\in \mc{W}_1$ given by $\upalpha_1=\xi\otimes \id$, we have 
\[
\upalpha_1\cdot\Phi_0=-4\xi \otimes \Phi_0\,.
\]
\noindent $\bullet$ \ For $\upalpha_2\in \mc{W}_2$ given by $\omega_{0}(\upalpha_2\cdot, \cdot)=\xi\otimes \hat{\omega}$, we have
\[
\upalpha_2\cdot\Phi_0=-4 \xi\otimes \sum_{a=1}^3 \hat{g}_{J_{a}}\odot g_{J_{a}}\,.
\]
\noindent $\bullet$ \ For $\upalpha_3\in \mc{W}_3$ given by $\omega_{0}(\upalpha_3\cdot, \cdot)=\xi\otimes \hat{g}_{J_b}$, we have
\[
\upalpha_3\cdot\Phi_0=-4\sum_{a\neq b}\xi\otimes(\hat{g}_{J_bJ_a}\odot g_{J_a})\,.
\]
\noindent $\bullet$ \ For $\upalpha_4\in \mc{W}_4$ given by $\omega_0(\upalpha_4 \cdot, \cdot)=\xi\otimes \rho(\cdot,J_b\cdot )$, we have
\[
\upalpha_4\cdot \Phi_0=4\xi\otimes (\rho \odot g_{J_b})\,.
\]
\end{lem}
\begin{proof}
The last claim is known from \cite[Section 1.3]{CGWPartI}, hence we only need to prove the rest of the cases. For the first one, and for any $x, y, z, w, u\in [\E\Hh]$ we obtain that
\begin{eqnarray*}
\upalpha_1\cdot\Phi_0(x, y, z, w)(u)&=&-\Phi_0(\upalpha_1 x, y, z, w)(u)-\Phi_0(x, \upalpha_1 y, z, w)(u)-\Phi_0(x, y, \upalpha_1 z, w)(u)\\
&&-\Phi_0(x, y, z, \upalpha_1 w)(u)\\
&=&- \Phi_0(\frac{1}{4n}\Tr_{2}(\upalpha_1)(u)x, y, z, w)-\Phi_0(x, \frac{1}{4n}\Tr_{2}(\upalpha_1)(u)y, z, w)\\
&&-\Phi_0(x, y, \frac{1}{4n}\Tr_{2}(\upalpha_1)(u)z, w)-\Phi_0(x, y, z, \frac{1}{4n}\Tr_{2}(\upalpha_1)(u)w)\\
&=&-\frac{1}{n}\Tr_{2}(\upalpha_1)(u)\Phi_0(x, y, z, w)\,.
\end{eqnarray*}
Similarly, for any $x, y, z, w, u\in [\E\Hh]$ we get
\begin{eqnarray*}
(\upalpha_2\cdot g_{J_a})(x, y)(u)&=&-g_{J_a}(\upalpha_2 x, y)(u)-g_{J_a}(x, \upalpha_2 y)(u)=-\omega_{0}(\upalpha_2 x, J_ay)(u)-\omega_{0}(\upalpha_2 y, J_{a}x)(u)\\
&=&-\xi(u)\hat{\omega}(x, J_{a}y)- \xi(u)\hat{\omega}(y, J_{a}x) \\
&=&-2\xi(u)\hat{\omega}(x, J_{a}y)=-2\xi(u)\hat{g}_{J_{a}}(x, y)\,,
\end{eqnarray*}
for any $a=1, 2, 3$. Then, the definition of $\Phi_0$ provides the claimed formula.
Let us now check the third case. For $a\neq b$ and for any $x, y, z, w, u\in [\E\Hh]$ we get
\begin{eqnarray*}
(\upalpha_3\cdot g_{J_a})(x, y)(u)&=&-g_{J_a}(\upalpha_3 x, y)(u)-g_{J_a}(x, \upalpha_3 y)(u)=-\omega_{0}(\upalpha_3 x, J_ay)(u)-\omega_{0}(\upalpha_3 y, J_ax)(u)\\
&=&-\xi(u)\hat{g}_{J_b}(x, J_ay)- \xi(u)\hat{g}_{J_b}(y, J_a x)=-\xi(u)\hat{\omega}(x, J_bJ_a y)- \xi(u)\hat{\omega}(y, J_bJ_a x) \\
&=&-2\xi(u)\hat{g}_{J_bJ_a}(x, y)\,.
\end{eqnarray*}
Finally, for $a= b$ and any $x, y, z, w, u\in [\E\Hh]$ we see that
\begin{eqnarray*}
(\upalpha_3\cdot g_{J_b})(x, y)(u)&=&-g_{J_b}(\upalpha_3 x, y)(u)-g_{J_b}(x, \upalpha_3 y)(u)=-\omega_{0}(\upalpha_3 x, J_by)(u)-\omega_{0}(\upalpha_3 y, J_bx)(u)\\
&=&-\xi(u)\hat{g}_{J_b}(x, J_by)- \xi(u)\hat{g}_{J_b}(y, J_bx)=-\xi(u)\hat{\omega}(x, J_b^2y)- \xi(u)\hat{\omega}(y, J^2_bx) \\
&=&0\,.
\end{eqnarray*}
Thus, again our assertion follows by the definition of $\Phi_0$.
\end{proof}

In the next step we will recover the elements $\upalpha_i\in \mc{W}_i$ from their action on $\Phi_0$. With this goal in mind it is useful to review how to produce a $(1,2)$-tensor by a $(0,5)$-tensor.
For such a procedure we need to consider the inverse of the scalar 2-form $\omega$, and of the induced (local) metrics $g_{J_a}, a=1, 2, 3$. Since there are two possible conventions how to define the inverse of an almost symplectic form, for the convenience of the reader we make this precise. 
\begin{defi}
Let $e_1,\dots, e_{2n},f_1,\dots, f_{2n}$ be a skew-Hermitian basis of $[\E\Hh]$. The inverse $\omega_0^{-1}$ of the standard scalar 2-form $\omega_0$ on $[\E\Hh]$ is expressed by
\[
\omega_0^{-1}:=\sum_{c=1}^{2n} e_c\otimes f_c-f_c\otimes e_c\in [S^2\E]\,.
\]
Moreover, for a $(0,2)$ tensor $A$, we define the following contractions 
\begin{eqnarray*}
A(\omega_0^{-1},\cdot)&:=&\sum_{c=1}^{2n}A(e_c,\cdot)f_c-A(f_c,\cdot)e_c\in \Ed([\E\Hh])\,,\\
A(\cdot,\omega_0^{-1})&:=&\sum_{c=1}^{2n}A(\cdot,f_c)e_c-A(\cdot,e_c)f_c\in \Ed([\E\Hh])\,,\\
A(\omega_0^{-1})&:=&\sum_{c=1}^{2n}A(e_c,f_c)-A(f_c,e_c)\in \R\,.
\end{eqnarray*}
\end{defi}
\noindent It is not hard to show that the inverse $\omega_0^{-1}$ and the above contractions are all independent of the choice of a skew-Hermitian basis.
Observe also that 
\begin{eqnarray*}
\omega_0(\cdot,\omega_0^{-1})&=&\omega_0(\omega_0^{-1},\cdot)=\Id\,,\\
g_{J_a}(\cdot,\omega_0^{-1})&=&-g_{J_a}(\omega_0^{-1},\cdot)=-J_a\,,\\
\omega_0(\cdot,g_{J_a}^{-1})&=&-\omega_0(g_{J_a}^{-1},\cdot)=J_a\,,\\
g_{J_a}(\cdot,g_{J_b}^{-1})&=&g_{J_a}(g_{J_b}^{-1},\cdot)=\begin{cases}
J_aJ_b\,,& a\neq b\,,\\
\Id\,,& a=b\,,
\end{cases}
\end{eqnarray*}
and moreover,
\begin{eqnarray*}
A(\omega_0^{-1})&=&\begin{cases}
0\,, & A\in [S^2\Hh]^*\oplus [\Lambda^{2}\E]^* \oplus [S^2_0\E]^*\oplus [S^2_0\E S^2\Hh]^*\oplus [\Lambda^2\E S^2\Hh]^*\,,\\
4n\,, & A=\omega_0\,.
\end{cases}\\
A(g_{J_a}^{-1})&=&\begin{cases}
0\,,& \langle \omega_0\rangle \oplus [\Lambda^{2}\E]^* \oplus [S^2_0\E]^*\oplus [S^2_0\E S^2\Hh]^*\oplus [\Lambda^2\E S^2\Hh]^*\,,\\
0\,,& A=g_{J_b}, a\neq b\,,\\
4n\,,& A=g_{J_a}\,.
\end{cases}
\end{eqnarray*}
Based on Lemma \ref{pureact} we can prove the following.
\begin{prop}\label{recov}
Suppose that the $(3,1)$-tensor $\hat h_0$ is given by
\[
\hat h_0:=\id \otimes \omega_0^{-1}+\sum_a J_a\otimes g_{J_a}^{-1}
\]
for some admissible basis $H=\{J_a\,,\ a=1,2,3\}$ of the standard quaternionic structure $Q_0$ on $[\E\Hh]$. Then, $\hat h_0$ does not depend on the choice of $H$. Moreover, the linear map
\[
{\sf c} : [\E\Hh]^*\otimes S^4[\E\Hh]^* \to \bigotimes^3 [\E\Hh]^*\,,\quad A\mapsto \sum_a A(x,J_ay,z,g_{J_a}^{-1})\,,\quad \forall \ x, y, z\in [\E\Hh]\,,
\]
defines a natural contraction of $\hat h_0$ with tensors of type $(0, 5)$. In particular,
\begin{eqnarray*}
{\sf c}(\upalpha_1\cdot\Phi_0)(x,y,z)&=&-8(2n-1)\omega_0(\upalpha_1(x,y),z),\\
{\sf c}(\upalpha_2\cdot\Phi_0)(x,y,z)&=&-8(n-1)\omega_0(\upalpha_2(x,y), z),\\
{\sf c}(\upalpha_3\cdot\Phi_0)(x,y,z)&=&\frac{16n}{3}\omega_0(\upalpha_3(x,y), z),\\
{\sf c}(\upalpha_4\cdot\Phi_0)(x,y,z)&=&-\frac{8(n+1)}{3}\omega_0(\upalpha_4(x,y), z).
\end{eqnarray*}
\end{prop}
\begin{proof}
The general formula for the natural contraction can be easily observed, when we express the application of the natural contraction of $\hat h_0$ with the $(0,5)$-tensor $\xi\otimes \omega_0\otimes \omega_0$ by
\[
\xi\otimes (\omega_0(\id x,y)\otimes \omega_0(z,\omega_0^{-1})+\sum_a\omega_0(J_a x,y)\otimes g_{J_a}(z,\omega_0^{-1}))\,,
\]
for some $\xi\in [\E\Hh]^*$. 
In particular, by using this formula we can verify that
\[
\omega_0(\id x,y)\otimes \omega_0(z,\omega_0^{-1})+\sum_a\omega_0(J_a x,y)\otimes g_{J_a}(z,\omega_0^{-1})=\omega_0(x,y)\otimes\id+\sum_a g_{J_a}( x,y)\otimes J_0=h_0\,,
\] 
which proves our first assertion. \\
Now, for $A\in [\E\Hh]^*\otimes S^4[\E\Hh]^*$ we can conclude that $A(x,y,y,\omega_0^{-1})=0$ for all $x,y,z\in [\E\Hh]^*$ and the claimed formula for ${\sf c}$ easily follows. 
In order to prove the rest of the presented formulas, we will compute $A(x,J_ay,z,g_{J_a}^{-1})$ for some $A=\xi\otimes (B\odot g_{J_b})\in [\E\Hh]^*\otimes S^4[\E\Hh]^*$, where $B\in S^2[\E\Hh]^*$. First, we see that
\begin{eqnarray*}
(\xi\otimes (B\odot g_{J_b}))(x,y,z,u,v)&=&\frac{\xi(x)}6(B(y,z)g_{J_b}(u,v)+B(y,u)g_{J_b}(z,v)+B(y,v)g_{J_b}(z,u)\\
&&+B(z,u)g_{J_b}(y,v)+B(z,v)g_{J_b}(y,u)+B(u,v)g_{J_b}(y,z))\,,
\end{eqnarray*}
for any $x,y,z,u,v\in [\E\Hh].$ Thus, by setting ${\sf{Z}}:=(\xi\otimes (B\odot g_{J_b}))(x,J_ay,z,g_{J_a}^{-1})$ we see that
\begin{eqnarray*}
{\sf Z}&=&\frac{\xi(x)}6(B(J_ay,z)g_{J_b}(g_{J_a}^{-1})+2B(J_ay,g_{J_b}(z,g_{J_a}^{-1}))\\
&&+2B(z,g_{J_b}(J_ay,g_{J_a}^{-1}))+B(g_{J_a}^{-1})g_{J_b}(J_ay,z))\\
&=&\begin{cases}
\frac{\xi(x)}6(-B(g_{J_a}^{-1})g_{J_aJ_b}(y,z)-2B(J_a y,J_aJ_bz)-2B(z,J_by)),& a\neq b\,,\\
\frac{\xi(x)}6(4nB(J_ay,z)+2B(J_ay,z)+2B(z,J_ay)+B(g_{J_a}^{-1})\omega_0(y,z)),& a= b\,.
\end{cases}
\end{eqnarray*}
Let us initially consider the case $B=g_{J_b}$. Then we obtain the following:
\begin{eqnarray*}
(\xi\otimes (B\odot g_{J_b}))(x,J_ay,z,g_{J_a}^{-1})&=&\begin{cases}
\frac{\xi(x)}3(-2\omega_0(y,z)),& a\neq b\,,\\
\frac{\xi(x)}3(4n+2)\omega_0(y,z),& a= b\,.
\end{cases}
\end{eqnarray*}
Therefore, having in mind the formulas from Lemma \ref{pureact} we result with
\begin{eqnarray*}
{\sf c}(\upalpha_1\cdot\Phi_0)(x,y,z)&=&-4\sum_{a,b=1}^3 \xi(x) \otimes(g_{J_b}\odot g_{J_b}) (J_ay,z,g_{J_a}^{-1})\\
&=&-4(4n+2-4)\xi(x)\omega_0(y,z)=-8(2n-1)\omega_0(\upalpha_1(x,y),z).
\end{eqnarray*}
Assume now that $B=\hat g_{J_b}$. For this case we compute
\begin{eqnarray*}
(\xi\otimes (B\odot \hat g_{J_b}))(x,J_ay,z,g_{J_a}^{-1})&=&\begin{cases}
-\frac{\xi(x)}32\hat \omega(y,z), & a\neq b\,,\\
\frac{\xi(x)}3(2n+2)\hat \omega(y,z),& a= b\,,
\end{cases}
\end{eqnarray*}
and consequently, in combination with Lemma \ref{pureact} we conclude that
\begin{eqnarray*}
{\sf c}(\upalpha_2\cdot\Phi_0)(x,y,z)&=&-4 \xi(x)\otimes \sum_{a,b=1}^3 (\hat{g}_{J_{b}}\odot g_{J_{b}})(J_ay,z,g_{J_a}^{-1})\\
&=&-4\xi(x)(2n+2-4)\hat{\omega}(y,z)=-8(n-1)\omega_0(\upalpha_2(x,y), z).
\end{eqnarray*}
We should also consider the case $B=\hat g_{J_c}, c\neq b$. Then we obtain
\begin{eqnarray*}
(\xi\otimes (B\odot \hat g_{J_b}))(x,J_ay,z,g_{J_a}^{-1})&=&\begin{cases}
0\,,&c\neq a\neq b\,,\\
\frac{\xi(x)}3(-2)\hat g_{J_aJ_b}(y,z)\,,&c= a\neq b\,,\\
\frac{\xi(x)}3(-2n-2)\hat g_{J_aJ_c}(y,z)\,,& a= b\,,
\end{cases}
\end{eqnarray*}
and therefore, for $J_b=J_1$ we compute
\begin{eqnarray*}
{\sf c}(\upalpha_3\cdot\Phi_0)(x,y,z)&=&-4 \xi(x)\otimes \sum_{a,c=1, c\neq 1}^3 (\hat{g}_{J_1J_{c}}\odot g_{J_{c}})(J_ay,z,g_{J_a}^{-1})\\
&=&-4 \xi(x)\otimes \sum_{a}^3 (\hat{g}_{J_3}\odot g_{J_{2}}-\hat{g}_{J_2}\odot g_{J_{3}})(J_ay,z,g_{J_a}^{-1})\\
&=&0-4\frac{\xi(x)}3(-2n-2)\hat g_{J_2J_3}(y,z)-4\frac{\xi(x)}3(-2)\hat g_{J_3J_2}(y,z)\\
&&-0+4\frac{\xi(x)}3(-2)\hat g_{J_2J_3}(y,z)+4\frac{\xi(x)}3(-2n-2)\hat g_{J_3J_2}(y,z)\\
&=&\frac{\xi(x)}3(8n+8-8-8+8n+8)\hat g_{J_1}(y,z)=\frac{16n}{3}\omega_0(\upalpha_3(x,y), z)\,.
\end{eqnarray*}
In fact, the same holds for $J_b=J_2,J_b=J_3$. 
Finally, let us we consider the case $B=\rho$. Then we get
\begin{eqnarray*}
(\xi\otimes (B\odot \hat g_{J_b}))(x,J_ay,z,g_{J_a}^{-1})&=&\begin{cases}
0,& a\neq b\,,\\
\frac{\xi(x)}3(-2n+2)\rho(y,J_az),& a= b\,,
\end{cases}
\end{eqnarray*}
and a direct computation proves our claim, i.e.,
\begin{eqnarray*}
{\sf c}(\upalpha_4\cdot\Phi_0)(x,y,z)&=&4 \xi(x)\otimes \sum_{a}^3 (\rho \odot g_{J_{b}})(J_ay,z,g_{J_a}^{-1})\\
&=&4\frac{\xi(x)}{3}(-2n-2)\rho(y,J_bz)=-\frac{8(n+1)}{3}\omega_0(\upalpha_4(x,y), z)\,.
\end{eqnarray*}
\end{proof}

Let us now explain how the above results become the key ingredients for an explicit construction of the map 
\[{\sf s}: \mathscr{E}\to (\mc{W}_1\oplus \mc{W}_2\oplus \mc{W}_3 \oplus \mc{W}_4)_{\mc{Q}}.\]
\begin{theorem}\label{smap}
The map
\[{\sf s}(\nabla \Phi):=-(\frac{\sf w_1}{8(2n-1)}+\frac{\sf w_2}{8(n-1)}-\frac{3\sf w_3}{16n}+\frac{3 \sf w_4}{8(n+1)})({\sf c}(\nabla \Phi))\]
satisfies ${\sf p}\circ {\sf s}=\Id_{\mathscr{E}}$. In particular, 
\[
\nabla+(\frac{\sf w_1}{8(2n-1)}+\frac{\sf w_2}{8(n-1)}-\frac{3\sf w_3}{16n}+\frac{3 \sf w_4}{8(n+1)})({\sf c}(\nabla \Phi))
\]
is a $\SO^*(2n)\Sp(1)$-connection.
\end{theorem}
\begin{proof}
The second claim is an immediate consequence of Proposition \ref{prophol} and the expression given in the first assertion. To prove the first statement, note that the covariant derivative $\nabla \Phi$ can be expressed as 
\[
\nabla\Phi=\sum_{a=1}^4 \upalpha_{a}\cdot \Phi\,,
\]
for some elements $\upalpha_{a}\in \mc{W}_a$ which themselves are sums of {\it pure} elements of $\mc{W}_a$, and $\cdot$ indicates their action. On the other hand, by linearity and by Proposition \ref{recov}, it follows that ${\sf s}(\nabla \Phi)= \sum_{a=1}^4 \upalpha_{a}$ and thus 
\[
{\sf p}\cc{\sf s}(\nabla \Phi)\rr=\sum_{a}\upalpha_a\cdot\Phi=\nabla\Phi\,,
\] 
i.e., ${\sf p}\circ{\sf s}=\Id_{\mathscr{E}}$. This completes the proof.
\end{proof}

Let us now decompose $\mc{W}_i$ into irreducible (for $n>3$) components and relate then with the image of the map $\delta$. Initially, we proceed to the level of representations.

\begin{prop}\label{irrWi}
The modules $\mc{W}_i$ $(i=1, 2, 3, 4)$ admit the following $\SO^*(2n)\Sp(1)$-equivariant decompositions:
\begin{eqnarray*}
\mc{W}_1&\cong& [\E\Hh]^*,\quad \mc{W}_2\cong [\K\Hh]^*\oplus [\E\Hh]^*\oplus [S^3_0\E\Hh]^*\,,\\
\mc{W}_3&\cong& [\K S^3\Hh]^*\oplus [\E S^3\Hh]^*\oplus [S^3_0\E S^3\Hh]^*\oplus [\K\Hh]^*\oplus [\E\Hh]^*\oplus [S^3_0\E\Hh]^*\,,\\
\mc{W}_4&\cong& [\K S^3\Hh]^*\oplus [\E S^3\Hh]^*\oplus [\Lambda^3\E S^3\Hh]^*\oplus [\K\Hh]^*\oplus [\E\Hh]^*\oplus [\Lambda^3\E\Hh]^*\,.
\end{eqnarray*}
Moreover, 
\begin{eqnarray*}
\bigoplus_a \mc{W}_a&\cap& \ke{\delta}\cong [\K S^3\Hh]^*\oplus [\E S^3\Hh]^*\oplus [S^3_0\E S^3\Hh]^*\oplus [\K\Hh]^*\oplus [\E\Hh]^*\oplus [S^3_0\E\Hh]^*\,,\\
\delta(\bigoplus_a \mc{W}_a)&\cap& \delta([\E\Hh]^*\otimes(\so^*(2n)\oplus \sp(1)))\cong [\E S^3\Hh]^*\oplus [\K\Hh]^*\oplus 2[\E\Hh]^*\oplus [\Lambda^3\E\Hh]^*\,,
\end{eqnarray*}
and consequently
\[
\delta(\bigoplus_a \mc{W}_a)/\delta\cc[\E\Hh]^*\otimes(\so^*(2n)\oplus \sp(1)\rr\cong \mc{X}_{12345}\,.
\]
\end{prop}
\begin{proof}
By \cite[\zeroprog]{CGWPartI} we know the decomposition of $\mc{W}_1$ and $\mc{W}_2$. Since 
\[
\mc{W}_1\oplus \mc{W}_2\oplus \mc{W}_4\cong \Lambda^2[\E\Hh]^*\otimes [\E\Hh]
\] we obtain the decomposition of $ \mc{W}_4$ by \cite[\genCa]{CGWPartI}. Note that if one replaces the antisymmetric tensors with trace-free symmetric tensors, then the decomposition of $\mc{W}_3$ is analogous to the decomposition of $\mc{W}_4$. Hence, by \cite[\genCa]{CGWPartI} and by the triviality of $\ke\CC\delta\big|_{[\E\Hh]^*\otimes\cc\so^*(2n)\oplus \sp(1)\rr}\RR$ we obtain
\begin{eqnarray*}
S^2[\E\Hh]^*\otimes [\E\Hh]&\cong&\mc{W}_3\oplus \CC [\E\Hh]^*\otimes\cc \so^*(2n)\oplus \sp(1)\rr\RR\\
&\cong& \mc{W}_3\oplus \delta \CC [\E\Hh]^*\otimes\cc \so^*(2n)\oplus \sp(1)\rr\RR \\
&\cong& [\K S^3\Hh]^*\oplus 2[\E S^3\Hh]^*\oplus [S^3_0\E S^3\Hh]^*\oplus 2[\K\Hh]^*\oplus 3[\E\Hh]^*\oplus [S^3_0\E\Hh]^*\oplus [\Lambda^3\E\Hh]^*
\end{eqnarray*}
and thus both of the claimed intersections follow. These also imply the final assertion.
\end{proof}

Let us emphasize on the fact that although $\bigoplus_a \mc{W}_a\cap \ke{\delta}\cong \mc{W}_3$, we have only $\ke{\delta}\cap \mc{W}_3= [S^3_0\E S^3\Hh]^*$ for the intersection. Nevertheless, we know that
\[
\ke{\delta}\cap \CC\mc{W}_3\oplus \cc [\E\Hh]^*\otimes(\so^*(2n)\oplus \sp(1))\rr\RR=S^3[\E\Hh]^*\cong [\E S^3\Hh]^*\oplus [\K\Hh]^*\oplus [S^3_0\E S^3\Hh]^*\oplus [\E\Hh]^*\,,
\]
which is the first prolongation of almost symplectic structures. Similarly,
\[
\ke{\delta}\cap \CC\mc{W}_1\oplus \mc{W}_2\oplus \cc[\E\Hh]^*\otimes(\so^*(2n)\oplus \sp(1))\rr\RR=[\E\Hh]^*\,,
\]
which is the first prolongation of almost quaternionic structures. This was described in our terms in \cite[\kernel]{CGWPartI} (see also \cite{Salamon86, AM}).

Having in mind that torsion-free connections are minimal $\Gl(4n, \R)$-connections, based on the Proposition \ref{irrWi} we result to the following characterization of the algebraic types of the corresponding $\SO^*(2n)$- and $\SO^*(2n)\Sp(1)$-geometries.
\begin{prop}
For a torsion-free connection $\nabla$, the following holds:
\begin{enumerate}
\item[$\mc{X}_1$] is the image of the components $2[\K S^3\Hh]^*\subset \mc{W}_{3}\oplus \mc{W}_{4}$ by $\delta$, and vanishes when the component of ${\sf s}(\nabla \Phi)$ in these two modules belongs to the kernel $\ke(\delta)$.
\item[$\mc{X}_2$] is the image of the component $[\Lambda^3\E S^3\Hh]^*\subset \mc{W}_{4}$ by $\delta$, and vanishes when this component of ${\sf s}(\nabla \Phi)$ vanishes.
\item[$\mc{X}_3$] is the image of the components $2[\K \Hh]^*\subset \mc{W}_{2}\oplus \mc{W}_{4}$ by $\delta$, and vanishes when the component of ${\sf s}(\nabla \Phi)$ in these two modules belongs to the kernel $\ke(\delta)$.
\item[$\mc{X}_4$] is the image of the components $2[\E \Hh]^*\subset \mc{W}_{2}\oplus \mc{W}_{4}$ by $\delta$, and vanishes when the component of ${\sf s}(\nabla \Phi)$ in these two modules belongs to the kernel $\ke(\delta)$.
\item[$\mc{X}_5$] is the image of the components $2[S^3_0\E \Hh]^*\subset \mc{W}_{2}\oplus \mc{W}_{3}$ by $\delta$, and vanishes when the component of ${\sf s}(\nabla \Phi)$ in these two modules belongs to the kernel $\ke(\delta)$.
\item[$\mc{X}_6$] is the image of the components $2[\E S^3\Hh]^*\subset \mc{W}_{3}\oplus \mc{W}_{4}$ by $\delta$, and vanishes when the component of ${\sf s}(\nabla \Phi)$ in these two modules belongs to the kernel $\ke(\delta)$.
\item[$\mc{X}_{47}$] is the image of the components $3[\E \Hh]^*\subset \mc{W}_{1}\oplus\mc{W}_{2}\oplus \mc{W}_{4}$ by $\delta$, and vanishes when the component of ${\sf s}(\nabla \Phi)$ in these three modules belongs to the kernel $\ke(\delta)$.
\end{enumerate}
\end{prop}
\begin{proof}
Clearly, for modules isomorphic to $\mc{X}_i$ that are not contained in $[\E\Hh]^*\otimes(\so^*(2n)\oplus \sp(1))$, the claim follows. For $\mc{X}_i$ that are contained in $[\E\Hh]^*\otimes(\so^*(2n)\oplus \sp(1))$, the claim follows by the discussion above. In particular, for the case of $\mc{X}_3$ we can exclude $[\K \Hh]^*\subset \mc{W}_3$ since $\delta([\K \Hh]^*)$ lies in $\delta\cc[\E\Hh]^*\otimes(\so^*(2n)\oplus \sp(1))\rr$. For $\mc{X}_4$, based on the same reason we exclude $[\E \Hh]^*\subset\mc{W}_3$.
\end{proof}
\begin{rem}\textnormal{
In general it is a hard task to derive the projections to the particular components. It is similarly challenging to indicate the elements of $[\E\Hh]^*\otimes(\so^*(2n)\oplus \sp(1))$ which remove the part of the torsion belonging to 
\[
\delta(\bigoplus_a \mc{W}_a)\cap\delta([\E\Hh]^*\otimes(\so^*(2n)\oplus \sp(1)))\,.
\]
We will carry out this problem in a later part of this series of papers devoted to $\SO^*(2n)$- and $\SO^*(2n)\Sp(1)$-structures. Such a ``removal'' would produce a minimal $\SO^*(2n)\Sp(1)$-connection for a particular normalization condition. As we will discuss below, the (minimal) connections described in \cite[\starconnectionsp]{CGWPartI} provide a particular example of this remarkable situation.}
\end{rem}
We now proceed with the application of Theorem \ref{smap} and Proposition \ref{irrWi} for an almost symplectic connection $\nabla=\nabla^{\omega}$. In this case, we see that 
\[
{\sf s}(\nabla^{\omega} \Phi):=\frac{3\sf w_3}{16n}({\sf c}(\nabla^{\omega} \Phi))\in \mc{W}_3\,,
\] 
and we can conclude the following.
\begin{corol}
For an almost symplectic connection $\nabla^{\omega}$ with torsion $T^{\omega}$, the following holds:
\begin{enumerate}
\item[$\mc{X}_1$] is the image of the component $[\K S^3\Hh]^*\subset \mc{W}_{3}$ by $\delta$, and vanishes when this component of ${\sf s}(\nabla^{\omega} \Phi)$ vanishes.
\item[$\mc{X}_2$] is the component $[\Lambda^3\E S^3\Hh]^*$ of $T^{\omega}$, and vanishes when this component of $T^{\omega}$ vanishes.
\item[$\mc{X}_3$] is the component $[\K \Hh]^*$ of $T^{\omega}$, and vanishes when this component of $T^{\omega}$ vanishes.
\item[$\mc{X}_4$] is the component $[\E \Hh]^*$ of $T^{\omega}$, and vanishes when this component of $T^{\omega}$ vanishes.
\item[$\mc{X}_5$] is the image of the component $[S^3_0\E \Hh]^*\subset \mc{W}_{3}$ by $\delta$, and vanishes when this component of ${\sf s}(\nabla \Phi)$ vanishes.
\item[$\mc{X}_6$] is the image of the components $[\E S^3\Hh]^*\subset \mc{W}_{3}$ by $\delta$, and vanishes when this component of ${\sf s}(\nabla \Phi)$ vanishes.
\item[$\mc{X}_{7}$] is the image of the components $[\E \Hh]^*\subset \mc{W}_{3}$ by $\delta$, and vanishes when this component of ${\sf s}(\nabla \Phi)$ vanishes.
\end{enumerate}
\end{corol}
Now, the application of Theorem \ref{smap} and Proposition \ref{irrWi} for $\nabla=\nabla^{H}$, and $\nabla=\nabla^{Q,\vol}$, respectively, should allow the recovering of the results obtained before. 
Indeed, as an immediate consequence of Theorem \ref{smap} we obtain:
\begin{corol} \textsf{1)} \ On an almost hs-H manifold $(M, H, \omega)$ the adapted connection $\nabla^{H, \omega}$ satisfies
$\nabla^{H,\omega}=\nabla^H-{\sf s}(\nabla^{H} \Phi)$, 
which is equivalent to saying that $\omega(-{\sf s}(\nabla^{H} \Phi)\cdot,\cdot)=\frac12\nabla^{H}\omega$. \\ 
\textsf{2)} \ On an almost qs-H manifold $(M, Q, \omega)$ the adapted connection $\nabla^{Q, \omega}$ satisfies
$\nabla^{Q,\omega}=\nabla^{Q,\vol}-{\sf s}(\nabla^{Q,\vol} \Phi)$, i.e., $-{\sf s}(\nabla^{Q, \vol} \Phi)=A^{\vol}$.
\end{corol}
\begin{proof}
For the first case Theorem \ref{smap} gives the expression
\[
{\sf s}(\nabla^{H} \Phi):=-(\frac{\sf w_1}{8(2n-1)}+\frac{\sf w_2}{8(n-1)})({\sf c}(\nabla \Phi))\in \mc{W}_{1}\oplus \mc{W}_{2}
\]
and the results follows by the uniqueness and the construction of the minimal connection $\nabla^{H,\omega}$. Similarly, for $\nabla=\nabla^{Q, \vol}$ 
by Theorem \ref{smap} we obtain
\[
{\sf s}(\nabla^{Q,\vol} \Phi):=-\frac{\sf w_2}{8(n-1)}({\sf c}(\nabla \Phi))\in \mc{W}_2
\] and one concludes as before, by the uniqueness and the construction of the minimal connection $\nabla^{Q,\omega}$.
\end{proof}

Let us finally indicate the application of Theorem \ref{smap} and Proposition \ref{irrWi} for $\nabla=\nabla^{Q}$. In this case we obtain 
\[
{\sf s}(\nabla^{Q} \Phi):=-(\frac{\sf w_1}{8(2n-1)}+\frac{\sf w_2}{8(n-1)})({\sf c}(\nabla \Phi))\in \mc{W}_{1}\oplus \mc{W}_{2}\,.
\]
However, as in \cite[\starconnectionsp]{CGWPartI} we can remove the component $[\E\Hh]^*$ in $\delta(\mc{W}_{1}\oplus \mc{W}_{2})\cap\delta([\E\Hh]^*\otimes(\so^*(2n)\oplus \sp(1)))$. Thus, the same results as for $\nabla^{Q,\omega}$ follow.


\section{Algebraic types via examples}\label{examples}

In this section, we present examples of almost hs-H structures on $\R^{4n}$ that realize particular pure types of torsion $\mc{X}_{i}$ of almost hs-H structures or the induced almost qs-H structures. These will be defined by a choice of a skew-Hermitian frame $u=(e_1,\dots,e_{2n},f_1,\dots,f_{2n})$, providing the trivialization $\R^{4n}\times \SO^*(2n)$ for the $\SO^*(2n)$-structure. The main reason for this is that the frame $u$ defines a $\SO^*(2n)$-connection $\nabla^u$ with torsion $T^{\nabla^u}=\dd\vartheta$ and curvature $R^{\nabla^u}=0$, where 
\[
\vartheta=(e_1^*,\dots,e_{2n}^*,f_1^*,\dots,f_{2n}^*)
\]
is the dual coframe to $u$. Therefore, one can determine the torsion type of the examples by projecting the torsion $T^{\nabla^u}$ to the normalization condition $\mc{D}(\so^*(2n))$ from \cite[\zeroprog]{CGWPartI}, along $\delta([\E\Hh]^*\otimes \so^*(2n))$. For this, one should also use the traces $\Tr_i$ $(i=1, \ldots, 4)$ introduced in \cite[Section 3.2]{CGWPartI}, and the maps $\pi_H, {\sf Alt}$ from \cite[Section 4.1]{CGWPartI}.

Since the spaces $\R^{4n}\cong \Hn^n= \R^n\otimes_{\R}\Hn$ and $\R^{4n}\cong \Hn^n= \C^n\otimes_{\C}\Hn$ (see \cite{Brocker}) carry flat hypercomplex structures, one can also think about starting with an (almost) symplectic structure on $\R^n$ or $\C^n$ and do a "quaternionification" of it. However, this is based on existence of global trivializations $T(\R^n)=\R^n\times \R^{n}$ and $T(\C^n)=\C^n\times \C^{n}$, respectively. Indeed, for any manifold $M$ admitting a global trivialization $TM=M\times \R^{n}$ or $TM=M\times \C^{n}$, there is a global trivialization $T(M^4)=M^4\times (\R^n)^4$ or $T(M^2)=M^2\times (\C^n)^2$, where $M^{k}$ means $M\times\cdots\times M$ ($k$-times). Hence, in this case a choice of isomorphism $(\R^n)^4\cong [\E\Hh]$ or $(\C^n)^2\cong [\E\Hh]$ induces a global skew-Hermitian frame. 

\begin{example}
Consider the following matrix group 
\[
M=\begin{pmatrix}\exp(y_1)& 0 \\y_2 & \exp(-y_1) \end{pmatrix}
\]
with $y_1,y_2\in \R$.
The left invariant vector fields $e_1=\partial_{y_1}+y_2\partial_{y_2}$ and $f_1=\exp(-y_1)\partial_{y_2}$ form a symplectic basis for the following left-invariant symplectic structure on $M$:
\[
\omega_M=\exp(y_1)(\dd y_1\wedge \dd y_2)\,.
\]
Let us also denote by $(x_1,\dots,x_8)$ the corresponding coordinates on $M^{4}$ and consider trivialization $T(M^4)=M^4\times \R^8$ given by the vector fields
\[
e_{1j}=\partial_{x_{2j-1}}+x_{2j}\partial_{x_{2j}}\,,\quad f_{1j}=\exp(-x_{2j-1})\partial_{x_{2j}}\,, \quad j=1,2,3,4\,.
\]
Consider the isomorphism $\R^8\cong [\E\Hh]$ induced by the following formula for a skew-Hermitian frame
\begin{gather*}
(e_{11}+f_{12}, e_{13}+f_{14}, e_{13}-f_{14}, -e_{11}+f_{12}, f_{11}-e_{12}, f_{13}-e_{14}, f_{13}+e_{14}, -f_{11}-e_{12})\,,
\end{gather*}
which preserves the property that the trivialization is provided by a symplectic frame of the symplectic structure.
Then, the corresponding almost hs-H structure on $M^{4}$ is expressed by
\begin{eqnarray*}
\omega&=&\sum_{j=1}^4 \frac{\exp(x_{2j-1})}{2}\dd x_{2j-1}\wedge \dd x_{2j}\,,\\
I&=& (\partial_{x_5}-(x_2\exp(x_1-x_5)-x_6)\partial_{x_6}) \dd x_1+\exp(x_1-x_5)\partial_{x_6} \dd x_2\\
&& -(\partial_{x_7}-(x_4\exp(x_3-x_7)-x_8)\partial_{x_8}) \dd x_3-\exp(x_3-x_7)\partial_{x_8} \dd x_4\\
&& - (\partial_{x_1}-(x_6\exp(x_5-x_1)-x_2)\partial_{x_2}) \dd x_5-\exp(x_5-x_1)\partial_{x_2} \dd x_6\\
&& +(\partial_{x_3}-(x_8\exp(x_7-x_3)-x_4)\partial_{x_4}) \dd x_7+\exp(x_7-x_3)\partial_{x_8} \dd x_8\,,\\
J&=&- (\partial_{x_3}-(x_2\exp(x_1-x_3)-x_4)\partial_{x_4}) \dd x_1-\exp(x_1-x_3)\partial_{x_4} \dd x_2\\
&&+(\partial_{x_1}-(x_4\exp(x_3-x_1)-x_2)\partial_{x_2}) \dd x_3+\exp(x_3-x_1)\partial_{x_2} \dd x_4\\
&&-(\partial_{x_7}-(x_6\exp(x_5-x_7)-x_8)\partial_{x_8}) \dd x_5-\exp(x_5-x_7)\partial_{x_8} \dd x_6\\
&&+(\partial_{x_5}-(x_8\exp(x_7-x_5)-x_6)\partial_{x_6}) \dd x_7+\exp(x_7-x_5)\partial_{x_6} \dd x_8\,,
\end{eqnarray*}
with $K=IJ$. Moreover, the torsion of $\nabla^{u}$ has the nice expression
\[
T^{\nabla^{u}}=2\big(\dd x_1\wedge\dd x_2\otimes \partial_{x_2}+\dd x_3\wedge\dd x_4\otimes\partial_{x_4}+\dd x_5\wedge\dd x_6\otimes \partial_{x_6}+\dd x_7\wedge\dd x_8\otimes \partial_{x_8}\big)\,.
\]
By construction $\dd\omega=0$, because $\dd \omega_M=0$. Hence by Theorem \ref{symplecticthem2} this example must be at least of type $\mc{X}_{1567}$. Moreover, none of the almost complex structures $\{I,J,K\}$ are integrable, and a computation based on the traces $\Tr_i$ and the map $\pi_H$, shows that $\mc{X}_5$ and $\mc{X}_6$ are trivial. Thus, $(M^4, I, J, K, \omega)$ is of type $\mc{X}_{17}$ and the induced almost qs-H structure is of pure type $\mc{X}_{1}.$
\end{example}

\begin{example}
Let us also present an example of pure type $\mc{X}_3$. On $\R^{12}=(x_1,\dots,x_{12})$ consider the following frame $u$ and its dual co-frame $\vartheta:$
\begin{eqnarray*}
u&=&(\partial x_1,\partial x_2+x_4\partial x_6+x_7 \partial x_9+x_{10}\partial x_{12}, \partial x_3, \partial x_4,\partial x_5-x_{1}\partial x_6-x_{10}\partial x_9+x_7\partial x_{12}, \partial x_6,\\
&&\partial x_7,\partial x_8+x_{10}\partial x_6-x_1\partial x_9-x_4\partial x_{12},\partial x_9, \partial x_{10}, \partial x_{11}-x_7\partial x_6+x_4\partial x_9-x_1\partial x_{12}, \partial x_{12})\,,\\
\vartheta&=&(\dd x_1, \dd x_2, \dd x_3, \dd x_4, \dd x_5, \dd x_6-x_4\dd x_2+x_1\dd x_5-x_{10}\dd x_8+x_7\dd x_{11}, \dd x_7, \dd x_8,\\
&& \dd x_9-x_7\dd x_2+x_{10}\dd x_5+x_1\dd x_8-x_4\dd x_{11}, \dd x_{10}, \dd x_{11}, \dd x_{12}-x_{10}\dd x_2-x_{7}\dd x_5\\
&& +x_{4}\dd x_8+x_{1}\dd x_{11})\,.
\end{eqnarray*}
Then, a computation shows that
\begin{eqnarray*}
\omega&=& \dd x_1\wedge \dd x_7+\dd x_2\wedge \dd x_8+\dd x_3\wedge (\dd x_9-x_7\dd x_2+x_{10}\dd x_5+x_1\dd x_8-x_4\dd x_{11})\\
&&+\dd x_4\wedge \dd x_{10}+\dd x_5\wedge \dd x_{11}+(\dd x_6-x_4\dd x_2+x_1\dd x_5-x_{10}\dd x_8+x_7\dd x_{11})\\
&&\wedge (\dd x_{12}-x_{10}\dd x_2-x_{7}\dd x_5+x_{4}\dd x_8+x_{1}\dd x_{11})\\
\dd\vartheta&=&(0, 0, 0, 0, 0, -\dd x_4\wedge \dd x_2+\dd x_1\wedge \dd x_5-\dd x_{10}\wedge \dd x_8+\dd x_7\wedge \dd x_{11},\\
&& 0, 0, -\dd x_7\wedge \dd x_2+\dd x_{10}\wedge \dd x_5+\dd x_1\wedge \dd x_8-\dd x_4\wedge \dd x_{11}, 0, 0,\\
&& -\dd x_{10}\wedge \dd x_2-\dd x_{7}\wedge \dd x_5+\dd x_{4}\wedge \dd x_8+\dd x_{1}\wedge \dd x_{11})\\
T^{\nabla^u}&=&(\dd x_2\wedge \dd x_4+\dd x_1\wedge \dd x_5+\dd x_{8}\wedge \dd x_{10}+\dd x_7\wedge \dd x_{11})\otimes \partial x_6\\
&&+(\dd x_2\wedge \dd x_7-\dd x_{5}\wedge \dd x_{10}+\dd x_1\wedge \dd x_8-\dd x_4\wedge \dd x_{11})\otimes \partial x_9\\
&&+(\dd x_{2}\wedge \dd x_{10}+\dd x_{5}\wedge \dd x_7+\dd x_{4}\wedge \dd x_8+\dd x_{1}\wedge \dd x_{11})\otimes \partial x_{12}\,.
\end{eqnarray*}
With the help of \cite[\kernel]{CGWPartI} we see that the traces $\Tr_1(T^{\nabla^u}),\dots,\Tr_4(T^{\nabla^u})$ vanish, and the intrinsic torsion components $\mc{X}_4,\mc{X}_6,\mc{X}_7$ are trivial. Although the maps ${\sf Alt}$ and $\pi_H$ do not commute, by applying them to $T^{\nabla^u}$ we conclude that this example is of pure type $\mc{X}_{3}=[\K\Hh]^*$.
\end{example}

\begin{example}
Finally, we describe an example where a conformal change of the symplectic structure is applied. On the open set $U\subset \R^{8}=(x_1,\dots,x_8)$ of $\R^8$ given by $x_1\neq 0$, consider the following frame $u$ and its dual co-frame $\vartheta:$
\begin{eqnarray*}
u&=&(\frac1{x_1}\partial x_1,\frac1{x_1}\partial x_2,\frac1{x_1} \partial x_3,\frac1{x_1} \partial x_4,\frac1{x_1}\partial x_5,\frac1{x_1}\partial x_7,\frac1{x_1}\partial x_8)\,,\\
\vartheta&=&(x_1\dd x_1, x_1\dd x_2, x_1\dd x_3, x_1\dd x_4, x_1\dd x_5, x_1\dd x_6, x_1 \dd x_7, x_1 \dd x_8)\,.
\end{eqnarray*}
Therefore, the scalar 2-form is given by $\omega=x_1^2\sum_{a=1}^{4}\dd x_a\wedge \dd x_{4+a}$ and we compute
\begin{eqnarray*}
\dd\vartheta&=&(0,\dd x_1\wedge \dd x_2, \dd x_1\wedge \dd x_3, \dd x_1\wedge \dd x_4, \dd x_1\wedge \dd x_5, \dd x_1\wedge \dd x_6, \dd x_1\wedge \dd x_7, \dd x_1\wedge \dd x_8)\,,\\
T^{\nabla^u}&=&\sum_{a=2}^8 \frac{1}{x_1}\dd x_1\wedge \dd x_a\otimes \partial x_a= \delta(\frac{1}{2x_1}\dd x_1\otimes \id)\,.
\end{eqnarray*}
By Remark \ref{confchang} we know that projections to $\mc{X}_4$ and $\mc{X}_7$ are proportional to $\frac{1}{2x_1}\dd x_1$, and thus in this example the corresponding $\SO^*(2n)$-structure is of type $\mc{X}_{47}$. In particular, the almost hypercomplex structure is hypercomplex, according to Theorem \ref{hypercomplexthm}, and the induced almost qs-H structure is of pure type $\mc{X}_{4}$.
\end{example}


\section{Homogeneous $\SO^*(2n)$- and $\SO^*(2n)\Sp(1)$-structures}\label{homogeneousex}
In this section we focus on $\SO^*(2n)$- and $\SO^*(2n)\Sp(1)$-structures which are invariant under the action of a Lie group $K$. In other words, we treat $4n$-dimensional (connected) homogeneous manifolds $K/L$ admitting a $K$-invariant $\SO^*(2n)$-structure or a $K$-invariant $\SO^*(2n)\Sp(1)$-structure (for details on homogeneous spaces we refer to \cite{Hel, CS}). 

\subsection{A characterization of invariant scalar 2-forms}
Let $K/L$ be a homogeneous space where $L\subset K$ is a closed Lie subgroup of $K$, and let us denote by $\chi : L\to\Aut(\fr{k}/\fr{l})$ the isotropy representation of $L$ on $\fr{k}/\fr{l}\cong T_{o}K/L$, where $o=eL\in K/L$ is the identity coset.
Recall by \cite[\firstjet]{CGWPartI} that hypercomplex/quaternionic symplectomorphisms are determined by their first jet. Therefore, to discuss invariant $\SO^*(2n)$- or $\SO^*(2n)\Sp(1)$-structures on a homogeneous space $K/L$ as above, 
without loss of generality we may assume that the isotropy representation $\chi$ is faithful. 
So we can identify the Lie algebra $\fr{l}$ of the stabilizer $L$ with the Lie algebra $\chi_{*}(\fr{l})$ of the linear isotropy group $\chi(L)\subset\Aut(\fr{k}/\fr{l})$.
Recall also that a homogeneous space $K/L$ is called \textsf{reductive} if there exists a $\Ad(L)$-invariant complement $\fr{m}$ of $\fr{l}$ in $\fr{k}$, that is, 
\[
\fr{k}=\fr{l}\oplus\fr{m}\,,\quad\text{and}\quad \Ad(L)\fr{m}\subset\fr{m}\,.
\]
The last relation implies the inclusion $[\fr{l}, \fr{m}]\subset\fr{m}$. When $K/L$ is reductive, we can identify $\fr{m}\cong T_{o}K/L$ and $\chi$ with $\Ad : L\to\Aut(\fr{m})$, i.e., $\chi\cong\Ad_{K}|_{L \times \fr{m}}$.

It is known that the geometric properties of $K/L$ can be examined by restricting our attention to the origin $o=eL$. In particular, $K$-invariant tensors on $K/L$ are in bijective correspondence with $\chi(L)$-invariant tensors on the tangent space $T_{o}K/L$. Thus, for example we see that 
\begin{enumerate}
\item A $K$-invariant almost hypercomplex structures corresponds to a linear hypercomplex structure $H=\{J_{a} : a=1, 2, 3\}$ on $\frak{k}/\frak{l}$ for which each $J_{a}\in\Ed(\fr{k}/\fr{l})$ commutes with the isotropy representation $\chi$.
\item A $K$-invariant almost quaternionic structure corresponds to a linear quaternionic structure $Q\subset\Ed(\frak{k}/\frak{l})$ on $\frak{k}/\frak{l}$ which is invariant (as a subspace) for the natural action of $K$.
\end{enumerate}
Consider a homogeneous space $K/L$ with a $K$-invariant almost hypercomplex structure $H$ or with a $K$-invariant almost quaternionic structure $Q$. Then it makes sense to consider the corresponding space $\Lambda^2_{\sc}(\frak{k}/\frak{l})^*$ of scalar 2-forms on $\frak{k}/\frak{l}$ with respect to $H_o$ or $Q_o$, respectively. Given such a 2-form $\omega_o$, we may use it to identify $\fr{k}/\fr{l}$ and its dual $(\fr{k}/\fr{l})^*$. Next we shall denote by 
$\Lambda^2_{\sc}(\frak{k}/\frak{l})^{\fr{l}}$ the space of $\chi(L)$-invariant scalar 2-forms on $M=K/L$ and by $\Omega_{\sc}^2(M)^{K}$ the corresponding smooth sections, i.e., smooth $K$-invariant scalar 2-forms on $M=K/L$. 
\begin{prop}
$\al)$ \ On a homogeneous space $M=K/L$ with a $K$-invariant almost hypercomplex structure $H=\{I, J, K\}$ the following are equivalent: 
\begin{enumerate}
\item[${\sf 1)}$] There is a $K$-invariant almost hs-H structure $(H, \omega)$ on $K/L$, i.e., $\omega\in\Omega_{\sc}^2(M)^{K}$;
\item[${\sf 2)}$] The linear scalar 2-form $\omega_o$ on $\fr{k}/\fr{l}\cong T_{o}K/L$ is invariant under the isotropy representation $\chi : L \to\Aut(\frak{k}/\frak{l})$, i.e., $\omega_o\in \Lambda^2_{\sc}(\frak{k}/\frak{l})^{\fr{l}}$;
\item[${\sf 3)}$] There is a basis of $\fr{k}/\fr{l}$ adapted to $H_o=\{I_o, J_o, K_o\}$ (in terms of \cite[\EHbases]{CGWPartI}) inducing a Lie group homomorphism $i : L\to \Gl(n, \Hn)\subset\Gl([\E\Hh])$, such that $(\dd i)(\fr{l})\subset\fr{so}^*(2n)\subset\fr{gl}([\E\Hh])$.
\end{enumerate}
$\beta)$ \ On a homogeneous space $K/L$ with a $K$-invariant almost quaternionic structure $Q$ the following are equivalent: 
\begin{enumerate}
\item[${\sf 1)}$] There is a $K$-invariant almost qs-H structure $(Q, \omega)$ on $K/L$, i.e., $\omega\in\Omega_{\sc}^2(M)^{K}$;
\item[${\sf 2)}$] The linear scalar 2-form $\omega_o$ on $\fr{k}/\fr{l}\cong T_{o}K/L$ is invariant under the isotropy representation $\chi : L \to\Aut(\frak{k}/\frak{l})$, i.e., $\omega_o\in \Lambda^2_{\sc}(\frak{k}/\frak{l})^{\fr{l}}$;
\item[${\sf 3)}$] For any admissible basis $H_o$ of $Q_o$, there is a basis of $\fr{k}/\fr{l}$ adapted to $H_o$ inducing a Lie group homomorphism $i : L\to \Gl(n, \Hn)\Sp(1)\subset\Gl([\E\Hh])$, such that $(\dd i)(\fr{l})\subset\fr{so}^*(2n)\oplus\fr{sp}(1)\subset\fr{gl}([\E\Hh])$.
\end{enumerate}
\end{prop}

Let us now discuss invariant adapted connections. 
\begin{corol}
The minimal $\SO^*(2n)$-connections $\nabla^{H,\omega}$ and $\SO^*(2n)\Sp(1)$-connections $\nabla^{Q,\omega}$ for $K$-invariant almost hs-H/qs-H structure on $K/L$, respectively, are $K$-invariant connections.
\end{corol}
\begin{proof}
Since hypercomplex/quaternionic symplectomorphisms map $\SO^*(2n)$- and $\SO^*(2n)\Sp(1)$-connections onto $\SO^*(2n)$- and $\SO^*(2n)\Sp(1)$-connections with related torsion, respectively, the claim follows from uniqueness of the minimal connections $\nabla^{H,\omega}$ and $\nabla^{Q,\omega}$.
\end{proof}

In order to capture the information about general invariant adapted connections, we need to consider the following generalization of a classical result of \cite{Wang}, see also \cite[Prop. 1.5.15]{CS}.
\begin{prop}\label{extension}
On a homogeneous space $K/L$ the following holds:\\
\noindent \textsf{1)} There is a $K$-invariant almost hs-H structure on $K/L$ together with $K$-invariant $\SO^*(2n)$-connection $\nabla$, if and only if there is Lie group homomorphism $i: L\to \SO^*(2n)$ and a linear map $\alpha: \fr{k}\to [\E\Hh]\oplus \so^*(2n)$ that restricts to a linear isomorphism $\fr{k}/\fr{l}\cong [\E\Hh]$ and satisfies 
\[
\alpha(Y)=(\dd i)(Y)\,,\quad \alpha(\Ad(\ell)(X))=\Ad(i(\ell))(\alpha(X))\,,
\]
for all $X\in \fr{k}$, $Y\in \fr{l}$ and $\ell\in L$.\\
\noindent \textsf{2)}There is a $K$-invariant almost qs-H structure on $K/L$ together with $K$-invariant $\SO^*(2n)\Sp(1)$-connection $\nabla$, if and only if there is Lie group homomorphism $i: L\to \SO^*(2n)\Sp(1)$ and a linear map $\alpha: \fr{k}\to [\E\Hh]\oplus \so^*(2n)\oplus \sp(1)$ that restricts to a linear isomorphism $\fr{k}/\fr{l}\cong [\E\Hh]$ and satisfies
\[
\alpha(Y)=(\dd i)(Y)\,,\quad \alpha(\Ad(\ell)(X))=\Ad(i(\ell))(\alpha(X))\,,
\]
for all $X\in \fr{k}$, $Y\in \fr{l}$ and $\ell\in L$.
\end{prop}
\begin{proof}
These claims follow for example by arguments as in \cite{CS}. Indeed, the tautological 1-form on the geometric structure together with the connection 1-form of the $K$-invariant connection $\nabla$, form a Cartan connection of affine type $([\E\Hh]\rtimes \SO^*(2n), \SO^*(2n))$ and 
\[
([\E\Hh]\rtimes \SO^*(2n)\Sp(1),\SO^*(2n)\Sp(1))\,,
\]
respectively. Conversely, the projections to $\so^*(2n)$ or $\so^*(2n)\oplus \sp(1)$ (along $[\E\Hh]$) of the restriction of the map $\alpha$ to $\fr{k}/\fr{l}$ induces (together with the Maurer-Cartan form on $K$) the $K$-invariant $\SO^*(2n)$-connection or $\SO^*(2n)\Sp(1)$-connection, respectively. 
Note that $K\times_i \SO^*(2n)$ or $K\times_i \SO^*(2n)\Sp(1)$ are the underlying invariant $\SO^*(2n)$-type structures, where the tautological 1-form is induced by the projection to $[\E\Hh]$ of the restriction of the map $\alpha$ to $\fr{k}/\fr{l}$.
\end{proof}
Let $\fr{g}$ be one of the Lie algebras $\fr{so}^*(2n)$ or $\fr{so}^*(2n)\oplus\fr{sp}(1)$, and let $(\alpha,i)$ be a pair satisfying the conditions of Proposition \ref{extension}. 
The linear map $\alpha: \fr{k}\to [\E\Hh]\oplus\fr{g}$ splits into the maps 
\[
\alpha_{[\E\Hh]}: \fr{k}\to [\E\Hh]\,, \quad \alpha_{\fr{g}}: \fr{k}\to \fr{g}\,.
\]
The first provides the isomorphism $\fr{k}/\fr{l}\cong [\E\Hh]$, while the second one is the so-called \textsf{Nomizu map}. Note that $\alpha_{[\E\Hh]}|_\fr{l}=0$ and $\alpha_{\fr{g}}|_\fr{l}=\dd i.$ Let $\nabla$ be a $K$-invariant connection on $K/L$ corresponding to the pair $(\alpha,i)$. The torsion $T^{\nabla}_{o}\in \Lambda^{2}[\E\Hh]^*\otimes[\E\Hh]$ of $\nabla$ is given by 
\[
T^{\nabla}_{o}(x, y)=\alpha_{\fr{g}}(\alpha_{[\E\Hh]}^{-1}(x))y-\alpha_{\fr{g}}(\alpha_{[\E\Hh]}^{-1}(y))x-\alpha_{[\E\Hh]}([\alpha_{[\E\Hh]}^{-1}(x), \alpha_{[\E\Hh]}^{-1}(y)]_{\fr{k}})\,,
\]
for any $x,y \in [\E\Hh]$.\\
Similarly, the curvature $R^{\nabla}_{o}\in \Lambda^2 [\E\Hh]^*\otimes \fr{g}$ of $\nabla$ 
has the form
\[
R^{\nabla}_{o}(x,y)=[\alpha_{\fr{g}}(\alpha_{[\E\Hh]}^{-1}(x)),\alpha_{\fr{g}}(\alpha_{[\E\Hh]}^{-1}(y))]_{\fr{g}}-\alpha_{\fr{g}}([\alpha_{[\E\Hh]}^{-1}(x), \alpha_{[\E\Hh]}^{-1}(y)]_{\fr{k}})\,,
\]
for any $x,y\in [\E\Hh]$.

On a homogeneous space $K/L$ with a reductive decomposition $\fr{k}=\fr{m}\oplus \fr{l}$, the description of $K$-invariant $\SO^*(2n)$- and $\SO^*(2n)$-structures in terms of Proposition \ref{extension} is simplified. In this case we use $\alpha_{[\E\Hh]}$ to identify $\fr{m}$ with $[\E\Hh]$ and thus restrict $\alpha_{\fr{g}}$ to a map $\alpha_{\nabla}: [\E\Hh]\to \fr{g}$. In particular,
\begin{eqnarray*}
T^{\nabla}_{o}(x, y)&=&\alpha_{\nabla}(x)y-\alpha_{\nabla}(y)x-[x,y]_{\fr{m}}\,,\\
R^{\nabla}_{o}(x,y)&=&[\alpha_{\nabla}(x),\alpha_{\nabla}(y)]_{\fr{g}}-\alpha_{\nabla}([x,y]_{\fr{m}})-(\dd i)([x,y]_{\fr{l}})\,,
\end{eqnarray*}
for any $x,y \in [\E\Hh]\cong\fr{m}$, where we split $[x,y]_{\fr{k}}=[x,y]_{\fr{m}}+ [x,y]_{\fr{l}}$ according to the reductive decomposition. 
\begin{rem}\textnormal{
Recall that there is a $K$-invariant $G$-connection $\nabla^u$ with $\alpha_{\nabla}=0$, called the \textsf{canonical connection} on $K/L$, which depends on the reductive complement $\fr{m}$. The torsion and curvature of $\nabla^u$ are given by (see \cite{KoNo})
\[
T^{\nabla^u}_{o}(x, y)=-[x,y]_{\fr{m}}\,,\quad\quad R^{\nabla^u}_{o}(x,y)=-(\dd i)([x,y]_{\fr{l}})
\]
respectively, for any $x,y\in \fr{m}\cong[\E\Hh]$. The canonical connection is the closest analogue of the connections $\nabla^u$ discussed in Section \ref{examples}, when we interpret $u$ as a local frame of 
\[
T\cc\exp(\fr{m})L\rr\subset T(K/L)
\]
induced by left invariant vector fields corresponding to a skew-Hermitian basis of $\fr{m}$. Indeed, the intrinsic torsion of such $\SO^*(2n)$- or $\SO^*(2n)\Sp(1)$-geometries is given by 
\[
p(T^{\nabla^u}_{o})=p(-[\cdot,\cdot]_{\fr{m}})\,.
\]
} 
\end{rem}

\begin{example}
Consider the homogeneous space $\Sl(4,\R)/\Sl(2,\R)$. We may identify the Lie algebra $\fr{sl}(2, \R)$ of $L=\Sl(2, \R)$ with the following matrices
\[
\fr{l}=\left\{\left( \begin {array}{cccc} { 0}&{ 0}&0&0
\\ \noalign{\medskip}{0}& 0&0&0
\\ \noalign{\medskip}0&0&{ l_1}&{ 
l_2}\\ \noalign{\medskip}0&0&{ l_3}&-{ l_1}\end {array} \right) : l_1,l_2,l_3\in \R\right\}\,.
\]
\noindent A reductive complement $\fr{m}$ of $\fr{l}$ in $\fr{k}$ is given by
\[
\fr{m}= \left\{\left( \begin {array}{cccc} { a_3}&{ a_9}&{ a_1}+{ 
a_5}+{ a_7}+ a_{11}&-{ a_1}-{ a_5}
\\ \noalign{\medskip}{ a_6}& a_{12}&-2\,{ a_1}-2\,
a_{11}&{ a_1}+{ a_5}-{ a_7}+a_{11}
\\ \noalign{\medskip}{ a_2}-{ a_4}-{ a_8}-a_{10}&{
a_8}+ a_{10}&-\frac12\,{ a_3}-\frac12\, a_{12}&{ 0
}\\ \noalign{\medskip}2\,{ a_2}-2\, a_{10}&-{ a_2}-{
a_4}+{ a_8}+ a_{10}&{0}&-\frac12\,{ a_3}-
\frac12\, a_{12}\end {array} \right)\right\}\,,
\]
with $a_1,\dots,a_{12}\in \R$. 
Define the map 
\[
\alpha_{[\E\Hh]}(A):=(a_1,\dots,a_{12})\,,
\]
where the right hand side are coordinates in an skew-Hermitian basis $\mc{B}$ of $[\E\Hh]$. It is not hard to check that $\fr{l}$ is mapped into $\fr{so}^*(6)$ via the map induced on endomorphisms by the map $\alpha_{[\E\Hh]} : \fr{sl}(4; \R)\to [\E\Hh]$. Moreover, it integrates to a Lie algebra homomorphism 
\[
i: \Sl(2,\R)\to \SO^*(6)\,.
\]
The explicit form of $i$ is given by a $12\times 12$ matrix, which we avoid to present. 
Clearly, the map $\alpha$ with components $\alpha|_{\fr{l}}=\dd i$ and $\alpha|_{\fr{m}}=\alpha_{[\E\Hh]}$, i.e.,
\[
\al=\alpha|_{\fr{l}}+\alpha|_{\fr{m}}=\dd i + \alpha_{[\E\Hh]}\,,
\]
together with the map $i$, satisfy the conditions of Proposition \ref{extension}. Therefore 
\begin{prop}
The homogeneous space $\Sl(4,\R)/\Sl(2,\R)$ carries a $\Sl(4,\R)$-invariant $\SO^*(6)$-structure. In particular, after the identification $\fr{m}\cong[\E\Hh]$ we have $H_{o}=H_0,\omega_{o}=\omega_0$ for the standard hypercomplex structure $H_0$ and standard symplectic form $\omega_0$ on $[\E\Hh]$. 
\end{prop}
Let $\nabla$ be the $K$-invariant connection on $K/L=\Sl(4,\R)/\Sl(2,\R)$ induced by the pair $(\alpha,i)$. Since $\alpha_{\nabla}=0$, the $K$-invariant $\SO^*(6)$-connection corresponding to the pair $(\alpha,i)$ is again the canonical connection $\nabla^u$. A direct computation of the torsion of $\nabla^{u}$ shows that the discussed $\SO^*(6)$-structure on $\Sl(4,\R)/\Sl(2,\R)$ is of generic type $\mc{X}_{1234567}$. 
\end{example}

\section{$\SO^*(2n)\Sp(1)$-structures and the bundle of Weyl structures }\label{Weylss}
The final section is about constructions providing examples of $\SO^*(2n)\Sp(1)$-manifolds, which arise within the context of \textsf{cotangent bundles} and \textsf{Weyl structures}. 
We begin with the following observation in the linear setting, which provides the main motivation for what it follows.

\begin{prop}\label{quat-doubledim}
	Let $(U, Q)$ be a quaternionic vector space. Then the contangent space $T^*U\cong U\times U^*$ admits a canonical linear qs-H structure. The quaternionic structure is the one induced by the natural action of $Q$, and the scalar 2-form $\omega$ is the canonical symplectic form given by the natural pairing $U\times U^*\to \R$.
	Under the identification of $U$ with the right quaternionic vector space $\Hn^n$, the map 
	\[
	\rho : \Hn^n\oplus (\Hn^n)^* \ni (u, \xi)\longmapsto (u, \bar{\xi}^{t})\in \Hn^{2n} 
	\]
	provides a natural quaternionic Darboux basis, see \cite[\darbas]{CGWPartI}.
\end{prop}
\begin{proof}
	The linear map $\rho$ induces a Lie group homomorphism $\Gl(n, \Hn)\Sp(1)\to \Gl(2n, \Hn)\Sp(1)$. It is a simple observation that $\rho_{*}(\fr{gl}(n, \Hn))\subset\fr{so}^*(4n)_{0}\subset\fr{so}^*(4n)$, where $\fr{so}^*(4n)_{0}$ is the reductive part of the $|1|$-grading of $\fr{so}^*(4n)$ presented in \cite[\gradbas]{CGWPartI}. We observe that in the bases provided by $\rho$ the natural linear quaternionic structure on $U\times U^*$ coincides with the standard quaternionic structure induced by the corresponding quaternionic Darboux basis on $T^*U$. Similarly, the natural pairing 
\[
((\cdot, \cdot)) :\Hn^n\oplus (\Hn^n)^*\to\R,\quad ((u, \xi))=\xi(u)\,,\quad \forall\ u\in \Hn^n, \xi\in (\Hn^n)^*\,,
\]
corresponds to the standard symplectic form obtained by the quaternionic Darboux basis (and hence it is a scalar 2-form with respect the induced linear quaternionic structure on $U\times U^*$).
\end{proof}

Below we show that there are two possible ways to generalize this result to the manifold setting. One of them is based on the canonical almost symplectic structure $\omega_{\sf{W}}$ which exists on the \textsf{bundle of Weyl structures} over an almost quaternionic manifold found in \cite{CM}. The other one relies on the \textsf{canonical symplectic structure} $\omega_{\sf{C}}$ which exists on the cotangent bundle of any manifold. 
Next our aim is to discuss conditions for existence of almost quaternionic structures $Q$ for which $\omega_{\sf{C}}$ and $\omega_{\sf W}$ become scalar 2-forms with respect to $Q$. 
A further aim is to proceed by comparing these two $\SO^*(2n)\Sp(1)$-structures, a procedure which allows us to derive a theorem in terms of the \textsf{${\sf P}$-tensor}, appearing in parabolic geometries.

\subsection{Weyl structures}
One of the methods used in geometry and theory of PDEs is to introduce new "jet variables" which make the geometric objects or differential equations under examination, linear (in these new variables). When one works in a coordinate free setting, the basic tool to relate the new variables with the initial manifold, is encoded by the notion of connections. The use of connections with torsion makes the background theory of semiholonomic jets unnecessarily complicated in applications, and in favourable situations, one can encode such derivatives
via Cartan connections.

The basic object in the theory of Cartan geometries is the \textsf{Cartan bundle}. That is, a principal bundle $\mathcal{G}\to N$ over a smooth manifold $N$, with structure group $P$, which should be viewed as the bundle of new "jet coordinates". Such a $P$ should be considered as the Lie group of "jets of transition maps between the jet coordinates". The actual assignment of the coordinates to points $u\in \mathcal{G}$ is provided by a Cartan connection of type $(G,P)$, that is an isomorphism $T_u \mathcal{G}\cong \fr{g}$ to the Lie algebra $\fr{g}$ of $G$. 
Of course, it is natural to require that this map is equivariant with respect to the "change of coordinates" given by $P$ and reproduces the fundamental vector fields of the $P$-action on $\mathcal{G}$. In these coordinates, the tangent space is identified with $\mathcal{G}\times_{\chi} \fr{g}/\fr{p}$, where $\chi$ denotes the isotropy representation of $G/P$. Set $G_0:=P/\ker(\chi)$. Then, the quotient $\mathcal{G}_0:=\mathcal{G}/\ker(\chi)$ defines a $G_0$-structure on $N$.
\begin{defi}
The \textsf{Weyl structures} of the Cartan geometry are $G_0$-equivariant sections $\mathcal{G}_0\to \mathcal{G}$ for a splitting $G_0\to P$ of the projection $P\to P/\ker(\chi)$.
\end{defi}
The Weyl structures are in a bijective correspondence with smooth sections of the bundle $\mathcal{G}/G_0$, which is therefore called the \textsf{bundle of Weyl structures} associated to the Cartan geometry (see \cite{CS03, CS, CM} for Weyl structures related to parabolic geometries).

Let us assume that there is a $G_0$-invariant decomposition 
\[
\fr{g}=\fr{g}_-\oplus \fr{g}_0 \oplus \fr{p}_+\,,
\]
where $\fr{g}_-\cong \fr{g}/\fr{p}$ and $\fr{p}_+$ is the Lie algebra of $\ker(\chi)$. Then, the pullback on $\mathcal{G}_0$ of a Cartan connection by a section $\mathcal{G}_0\to \mathcal{G}$ (Weyl structure), decomposes according to its image to
\begin{itemize}
\item the component with values in $\fr{g}_-$, which is the tautological 1-form defining the underlying $G_0$-structure.
\item the component with values in $\fr{g}_0$, which is the connection 1-form of a $G_0$-connection.
\item the component with values in $\fr{p}_+$, which is usually called the ${\sf P}$-tensor, see \cite{CS03, CS}.
\end{itemize}
\begin{rem}
\textnormal{In the case of almost quaternionic geometries, $\mathcal{G}_0$ corresponds to the $\Gl(n,\Hn)\Sp(1)$-structure and a Weyl structure is equivalent to an Oproiu connection. This is because the bundle $\mathcal{G}$, viewed as a subbundle of the second order semiholonomic frame bundle, is spanned by all of the Oproiu connections. Moreover, the ${\sf P}$-tensor carries the remaining information about the Cartan connection (which is uniquely determined by the other data as we review later).}
\end{rem}

\subsection{The use of the bundle of Weyl structures}
Let $G={\sf P}\Gl(n+1,\Hn)$ be the projective linear group satisfying $\Sl(n+1, \Hn)/\Z_2\cong {\sf P}\Gl(n+1,\Hn)$, and $P$ be the parabolic subgroup stabilizing the quaternionic line spanned by the first basis vector in $\Hn^{n+1}$.
Recall that the quaternionic projective space is the homogeneous space $N=\Hn {\sf P}^n = G/P$, which serves as the flat model of the associated Cartan geometry.
A Levi subgroup of $P$ is given by $G_0=\Gl(n,\Hn)\Sp(1)$ and we have the following $G_0$-invariant decomposition into $\fr{g}_0$-modules: 
\[
\fr{g}=\sl(n+1,\Hn)=\g_-\oplus \fr{g}_0\oplus \p_+\,, \quad T_{eP}G/P\cong \fr{g}_{-}\,,\quad \fr{p}=\fr{g}_0\oplus \p_+\,.
\] 
The Lie group $G$ coincides with the total space of the Cartan bundle $\mc{G}\to N$ over $N$ and the Cartan connection is provided by the Maurer-Cartan form of $G$. That is, we have a canonical trivialization of the tangent bundle
	\[
	TG=G\times \sl(n+1,\Hn)\,.
	\]
Therefore, the homogeneous space $G/G_0$ coincides with the bundle of Weyl structures over $G/P=\Hn {\sf P}^n$ and in particular the pseudo-Wolf space 
\[
\widetilde{M}:=\Sl(n+1,\mathbb{H})/(\Gl(1,\mathbb{H})\Sl(n,\mathbb{H}))
\]
covers $M=G/G_0={\sf P}\Gl(n+1,\Hn)/\Gl(n,\Hn)\Sp(1)$. Note that locally, this is the analogue of the linear construction provided by Proposition \ref{quat-doubledim}. This is because $\exp(\fr{g}_{-})$ has an open orbit in $N$ which coincides with the quaternionic vector space $\fr{g}_{-}\cong T_x\Hn {\sf P}^n$.
Therefore, as a corollary of Proposition \ref{quat-doubledim} and \cite[\homogthem]{CGWPartI} we state the following
\begin{prop}\label{quatertangentmodel}
	The $8n$ dimensional homogeneous space $M=G/G_0$ has a natural torsion-free $G$-invariant qs-H structure. In particular, the invariant scalar 2-form $\omega_{\sf W}$ is induced by the pairing between $\g_-$ and $\p_+$, provided by the Killing form of $\sl(n+1,\Hn)$, that is $\g_-\cong [\E\Hh]$ and $\p_+\cong [\E\Hh]^*$.
\end{prop}
\begin{proof}
The $G_0$-equivariant projection $\pi:TG\to TM$ yields the decomposition
	\[
		TM=G\times_{G_0} \cc\sl(n+1,\Hn)/\fr{g}_0\rr=G\times_{G_0} (\g_-\oplus \p_+)\,,
	\]
and so $TM$ inherits an almost quaternionic structure from the canonical linear quaternionic structures on $\g_-\cong [\E\Hh]$ and $\p_+\cong [\E\Hh]^*$. In fact, as follows from \cite[\homogthem]{CGWPartI} this is an invariant quaternionic structure, since the covering $\tilde{M}\to M$ is a quaternionic symplectomorphism. 
\end{proof}

Next we shall generalize the above result by providing a universal construction within the category of non-integrable $\SO^*(2n)\Sp(1)$-structures. 
This lets us work with a general almost quaternionic manifold $(N, Q)$, instead of the flat model $\Hn {\sf P}^n$.
For this goal it is useful to construct the Cartan connection from local data first, i.e., a (local) quaternionic co-frame 
\[
\nu: TN\to \g_-\cong [\E\Hh]\,.
\]
So, let us assign to $(N, Q)$ the Cartan bundle $\mathcal{G}\rightarrow N$ together with a normal Cartan connection. Note that locally, a choice of a co-frame $\nu$ as above provides the local trivialization 
\[
\mathcal{G}=N\times G_0\times \exp(\p_+)\,.
\]
The pullback on $N$ of the Cartan connection is the following matrix of 1-forms on $N$:
\[
\left(\begin{matrix}
\gamma_i\nu^i& {\mbox{\sf P}}_i\nu^i\\
\nu& \Gamma_i\nu^i
\end{matrix}\right)\in \left(\begin{matrix}
\sp(1)+\R & \p_+\\
\g_-& \gl(n,\Hn)
\end{matrix}\right)\,.
\]
Here, the sum $\gamma_i+ \Gamma_i: TN\to \fr{g}_0$ can be viewed as the connection 1-form of a Oproiu connection $\nabla^{Q}$ and ${\sf P}$ is the ${\sf P}$-tensor in $T^*N\otimes T^*N$. Up to change of Oproiu connection, these are uniquely determined by the following normalization condition::
\[
\partial^*\kappa:=\sum_i 2\{Z_i,\kappa(.,X_i)\}=0\,,
\]
where $\kappa: N\to \Lambda^2 \g_-^*\otimes \sl(n+1,\Hn)$ is the curvature of the Cartan connection, that is
\[
\kappa:=\dd\left(\begin{smallmatrix}
\gamma_i\nu^i& {\mbox{\sf P}}_i\nu^i\\
\nu& \Gamma_i\nu^i
\end{smallmatrix}\right)+\{\left(\begin{smallmatrix}
\gamma_i\nu^i& {\mbox{\sf P}}_i\nu^i\\
\nu& \Gamma_i\nu^i
\end{smallmatrix}\right),\left(\begin{smallmatrix}
\gamma_i\nu^i& {\mbox{\sf P}}_i\nu^i\\
\nu& \Gamma_i\nu^i
\end{smallmatrix}\right)\}\,.
\]
Here, $\{ \ \, \ \}$ is the Lie bracket in $\sl(n+1,\Hn)$ and $\{X_i\}, \{Z_i\}$ are dual bases of $\g_-,\p_+$, respectively. The requirement of $\gamma_i$ taking values in $\sp(1)$, together with the normalization condition $\partial^*\kappa=0$, assigns to the co-frame $\nu$ a {\it unique} unimodular Oproiu connection and a unique symmetric ${\sf P}$-tensor. This provides the construction of the Cartan connection starting with the co-frame $\nu$, since at the point 
\[
(x, g_0,\exp(Z))\in N\times G_0\times \exp(\p_+)=\mathcal{G}
\] the Cartan connection consists of two parts: namely of the $\Ad(g_{0}\exp(Z))^{-1}$-action on the above pullback, and of the Maurer-Cartan form on $P=G_0\times \exp(\p_+).$

Now, based on a local description one can observe that the choice of an Oproiu connection $\nabla^{Q}$ provides a global trivialization $\mathcal{G}= \mathcal{G}_0\times \exp(\p_+)$. Then the pullback on $\mc{G}_0$ decomposes as above, which means that the ${\sf P}$-tensor is globally defined. In particular,
\begin{corol}
The bundle of Weyl structures $\mathcal{G}/G_0$ over $N$ trivializes to $\mathcal{G}_0\times_{G_0} \exp(\p_+)$. 
\end{corol}
Clearly, sections of this bundle are in bijective correspondence with Oproiu connections on $N$. The formulas for how these trivializations (and the corresponding pullbacks) change with the change of the Oproiu connection, can be found in \cite{CS}. However we will not need them explicitly.

The main advantage of the use of the Cartan connections for the description of almost quaternionic structures is that it immediately gives rise to the following result.

\begin{theorem}\label{weylstr}
The construction from Proposition \ref{quatertangentmodel} extends to a functor from the category of almost quaternionic manifolds to the category of almost qs-H manifolds. This functor restricts to a functor from the category of quaternionic manifolds to the category of almost qs-H manifolds with symplectic scalar 2-form $\omega_{\sf{W}}$.
\end{theorem}
\begin{proof}
	Let $(N,Q)$ be an almost quaternionic manifold. This admits a canonical Cartan connection on the $P$-bundle $\mathcal{G}\rightarrow N$, which provides the trivialization
\[
T\mathcal{G}=\mathcal{G}\times \sl(n+1,\Hn)\,.
\]
We have $M=\mathcal{G}/G_0$ and
\[
	TM=\mathcal{G}\times_{G_0} \sl(n+1,\Hn)/\fr{g}_0=\mathcal{G} \times_{G_0} (\g_-\oplus \p_+)\,,
\]
where the equality does not depend on the Levi factor $G_0$. 
Therefore, as in Proposition \ref{quatertangentmodel}, we obtain an almost qs-H structure on $M$. The Cartan connection and the corresponding trivialization is preserved by quaternionic automorphisms. Therefore, this construction is a functor from the category of almost quaternionic manifolds to the category of almost qs-H manifolds. Then, the second claim follows from the first one and \cite[Theorem 3.1]{CM}.
\end{proof}

Let us now illustrate Theorem \ref{weylstr} via an explicit example, by using an almost quaternionic manifold.

\begin{example}\label{jan1}
Consider $N=\R^8=(x_1,\dots,x_8)$ endowed with the quaternionic co-frame
\[
\nu=(\dd x_1,\dd x_2,\dd x_3,\dd x_4,\dd x_5+x_1 \dd x_2,\dd x_6,\dd x_7,\dd x_8)
\]
providing an isomorphism $T_{x}\R^8\to[\E\Hh]$, and let us denote by $Q$ the induced almost quaternionic structure on $N$.
Following the above construction, we obtain the following matrix of 1-forms on $N$ for the pullback of the Cartan connection along $\nu$:
\[
C:=\begin{pmatrix}
0&0& 0\\
\dd x_1+\dd x_2 i+\dd x_3 j+\dd x_4 k&0 &0 \\
\dd x_5+x_1 \dd x_2+\dd x_6 i+\dd x_7 j+\dd x_8 k& \frac16(2\dd x_2+ 2\dd x_1 i -\dd x_4 j+\dd x_3 k) & 0
\end{pmatrix}\,.
\]
Therefore, on the trivialization $N\times \exp(\fr{p}_+)=\R^8\times \fr{p}_+=(x_1,\dots,x_8,p_1,\dots,p_8)$, there is the following co-frame provided by the pullback of the Cartan connection along $\nu$ and the action of $\exp(\fr{p}_+)$:
{\tiny
\begin{gather*}
\begin{pmatrix}
1&-\bar{p}&-\bar{q}\\
0&1 &0 \\
0& 0 & 1
\end{pmatrix}(C+\dd \exp(\fr{p}_+))\begin{pmatrix}
1&\bar{p} &\bar{q}\\
0&1 &0 \\
0& 0 & 1
\end{pmatrix}
= 
\begin{pmatrix}
*&-(\bar{p}C_{21}+\bar{q}C_{31})\bar{p}-\bar{q}C_{32}+\dd \bar{p}& -(\bar{p}C_{21}+\bar{q}C_{31})\bar{q}+\dd\bar{q}\\
C_{21}&* &* \\
C_{31}& * & *
\end{pmatrix}
\end{gather*}}
where $\bar{p}=p_1-p_2 i-p_3 j-p_4 k$ and $\bar{q}= p_5-p_6 i-p_7 j-p_8 k.$

By using the above expressions, we can now provide the quaternionic Darboux basis (see \cite[\darbas]{CGWPartI}) via the isomorphism 
{\small\[
T(N\times \exp(\fr{p}_+))\to \Hn^4: (C_{21},C_{31},\dd p-(p\bar{C}_{21}+q\bar{C}_{31})p-q\bar{C}_{32}, \dd q -(p\bar{C}_{21}+q\bar{C}_{31})q)\,.
\]}
In the quaternionic Darboux basis the almost quaternionic structure $Q$ corresponds to right multiplication by $-i,-j,-k$, and the scalar 2-form is expressed by
\begin{eqnarray*}
\omega&=&-2p_5\frac13 \dd x_1\wedge \dd x_2+\frac16p_8 x_1\wedge \dd x_3 -\frac16p_7 x_1\wedge \dd x_4 +x_1\wedge \dd p_1+\frac16p_7 \dd x_2\wedge \dd x_3\\
&&+\frac16p_8 \dd x_2\wedge \dd x_4+\dd x_2\wedge \dd p_2+x_1 \dd x_2\wedge \dd p_5-\frac13p_5\dd x_3\wedge \dd x_4+\dd x_3\wedge \dd p_3\\
&&+\dd x_4\wedge \dd p_4+\dd x_5\wedge \dd p_5+\dd x_6\wedge \dd p_6+\dd x_7\wedge \dd p_7+\dd x_8\wedge \dd p_8\,.
\end{eqnarray*}
Let us point out that $\dd \omega\neq 0$ and that the quaternionic structure is not torsion-free, i.e., $T^{Q}\neq 0$. However, the (intrinsic) torsion of this example is too complicated to be discussed in full details here.
\end{example}

\subsection{The use of the cotangent bundle}
Let us now proceed with the second approach based on the cotangent bundle, and the observation that 
\[
\mathcal{G}_0\times_{G_0} \exp(\p_+)
\]
is naturally diffeomorphic to $T^*N$. Locally, this diffeomorphism is given by the map 
\[
{\sf t}: N\times \exp(\p_+)\longrightarrow T^*N\,,\quad {\sf t}(x, \exp(Z))= (x,\nu^i(x)Z_i)\,,
\]
for the local co-frame $\nu$. Therefore, 
\[
\mathcal{G} \times_{G_0} (\g_-\oplus \p_+)=TM\cong TT^*N\,,
\]
where the second isomorphism above depends only on the Oproiu connection $\nabla$, corresponding to the almost quaternionic structure $Q$ induced by $\nu$. In other words, we have two (almost) symplectic forms on $TT^*N$:
\begin{enumerate}
\item the canonical symplectic form $\omega_{\sf C}$ on $T^*N$,
\item the pullback $\omega_{\sf W, \nabla}={\sf t}^*\omega_{\sf{W}}$ on $T^*N$ of the almost symplectic form $\omega_{\sf{W}}$ on $M$ provided by $\nabla$.
\end{enumerate}
Moreover, we have two (almost) quaternionic structures:
\begin{enumerate}
\item[(a)] the quaternionic structure $Q_{\hor}$ induced by the horizontal distribution of $\nabla$ on $T^*N$,
\item[(b)] the pullback $Q_{\pul}={\sf t}^*Q_{\sf W}$ on $T^*N$ of the almost quaternionic structure $Q_{\sf W}$ on $M$ provided by $\nabla$.
\end{enumerate}
We obtain the following compatibility statements.

\begin{theorem}\label{cotangent-example}
Under the above assumptions the following claims hold:
\begin{enumerate}
\item[${\sf 1)}$] The (almost) symplectic forms $\omega_{\sf C}$ and $\omega_{{\sf W},\nabla}$ on $T^*N$ coincide, if and only if $(N,Q)$ is quaternionic and the Oproiu connection $\nabla$ is unimodular. The unimodularity condition is equivalent to say that $\nabla$ preserves some volume form and the corresponding ${\sf P}$-tensor is symmetric. 
\item[${\sf 2)}$] The (almost) quaternionic structures $Q_{\hor}$ and $Q_{\pul}$ on $T^*N$ coincide, if and only if ${\mbox{\sf P}}$ has values in $[\E\E]^*.$
\item[${\sf 3)}$] The pair $(Q_{\pul},\omega_{W,\nabla})$ defines an almost quaternionic skew-Hermitian structure on $T^*N.$
\item[${\sf 4)}$] The pair $(Q_{\hor},\omega_{W,\nabla})$ defines an almost quaternionic skew-Hermitian structure on $T^*N$, if and only if
\[
0={\mbox{\sf P}}(X,Y)-{\mbox{\sf P}}(Y,X) +{\mbox{\sf P}}(\J Y, \J X) -{\mbox{\sf P}}(\J X,\J Y)\,, 
\]
for all $X,Y\in \Gamma(TN)$ and $\J\in \Gamma(Z)$. This is equivalent to say that either $\nabla$ is unimodular, or ${{\sf P}}$ has values in $[\E\E]^*.$
\item[${\sf 5)}$] The pair $(Q_{\pul},\omega_{C})$ defines a smooth almost quaternionic skew-Hermitian structure on $T^*N$, if and only if $(N,Q)$ is quaternionic and either $Q_{\pul}=Q_{\hor}$, or $\omega_{C}=\omega_{W,\nabla}$.
\item[${\sf 6)}$] The pair $(Q_{\hor},\omega_{C})$ defines a smooth almost quaternionic skew-Hermitian structure on $T^*N$, if and only if $(N,Q)$ is a quaternionic manifold.
\end{enumerate}
\end{theorem}
\begin{proof}
Let us choose a local quaternionic frame $\nu$. Then, the isomorphism $T_{(x,\exp(Z))}M\cong\g_-\oplus \p_+$ has the form 
\[
\Ad(\exp(Z))^{-1}\left(\begin{smallmatrix}
\gamma_i\nu^i& {\mbox{\sf P}}_i\nu^i\\
\nu& \Gamma_i\nu^i
\end{smallmatrix}\right)=\left(\begin{smallmatrix}
\gamma_i\nu^i-Z_i\nu^i& {\mbox{\sf P}}_i\nu^i+\gamma_i\nu^iZ-Z_j \Gamma^j_i\nu^i-Z_i\nu^i Z \\
\nu& \Gamma_i\nu^i+\nu Z
\end{smallmatrix}\right)\,.
\] 
Let $(x^i)$ be local coordinates on $N$ and $(x^i,p_i)$ be the induced coordinates on $T^\ast N$. Then $\nu^i=\nu^i_j\dd x^j$ and thus thus we deduce that $N\times \exp(\p_+)\longrightarrow T^*N$ has the form
\[
(x^i,Z_i)\longmapsto (x^i,\nu_i^jZ_j)\,.
\]
Thus, 
\[
\partial_{x_i}-\sum_j({\mbox{\sf P}}_{kj}\nu^k+\gamma_k\nu^kZ_j-Z_l \Gamma^l_{kj}\nu^k-Z_k\nu^k Z_j)\partial_{Z_j}
\]
are the generators of $\mathcal{G} \times_{G_0} \g_-$ which we want to push forward to $T^*N$. The pushforward of $\partial_{x^i}$ is given by $\partial_{x^i}+\sum_j\frac{\partial \nu^k_j}{\partial x^i}Z_k\partial_{p_j}$ and the pushforward of $\partial_{Z_i}$ is equal to $\sum_j\nu^i_j\partial_{p_j}$. Note also that $Z_i=(\nu^{-1})_i^jp_j$. Altogether, for the pushforward of $\mathcal{G} \times_H \g_-$ we compute
\[
\partial_{x^i}+\sum_j\frac{\partial \nu^k_j}{\partial x^i}(\nu^{-1})_k^lp_l\partial_{p_j}-\sum_l({\mbox{\sf P}}_{kj}\nu^k_i+\gamma_k\nu^k_i(\nu^{-1})_j^lp_l-(\nu^{-1})_l^ap_a \Gamma^l_{kj}\nu^k_i-(\nu^{-1})_l^ap_a\nu^l_i (\nu^{-1})_j^kp_k)\nu^j_l\partial_{p_l}\,.
\]
Let us now use the same notation for the ${{\sf P}}$-tensor in the $(x^i,p_i)$ coordinates. Then, the above expression can be rewritten as 
\[
\partial_{x^i}-\sum_j{\mbox{\sf P}}_{ij}\partial_{p_j}+\theta(\nabla^{\frac12p_i}_{\partial_{x^i}} \partial_{x^k})\partial_{p_k}\,,
\]
where $\nabla^{\frac12p_i}$ is the Oproiu connection differing by $\frac12p_i$ (viewed as 1-form) from $\nabla$, and $\theta = p_i \dd x^i$ is the Liouville form on $TT^*N$. The precise formula for the change of Oproiu connection by a 1-form follows the conventions of \cite[Section 5.1.6]{CS}.

\noindent Let us also mention that the horizontal subspace corresponding to $\nabla$ is spanned by vector fields of the form
\[
\partial_{x^i}+\theta(\nabla_{\partial_{x^i}} \partial_{x^k})\partial_{p_k}.
\] 
We are ready now to prove the claims of the theorem. For the first claim, we compare the horizontal subspaces with the canonical symplectic form 
\[
\omega_{\sf{C}} = -\dd\theta = \dd x^i \wedge \dd p_i\,.
\]
We conclude that $\omega_{\sf{C}}=\omega_{{\sf W},\nabla}$, if and only both $\nabla$ is torsion-free (and thus also $\nabla^{\frac12p_i}$) and moreover the ${{\sf P}}$-tensor is symmetric, that is $(N,Q)$ is a unimodular quaternionic manifold. \\
For the second claim, one should compare the horizontal subspaces one to each other, and conclude that the ${{\sf P}}$-tensor provides all the difference between the quaternionic structures $Q_{\hor}$ and $Q_{\pul}$. Consequently, the second claim follows because elements of $[\E\E]^*$ are quaternionic linear. Now, the third claim clearly holds. 

For the fourth claim, set
\[
\tilde \J(X):=\J(X)+{\mbox{\sf P}}(\J X)-\J {\mbox{\sf P}}(X)\,,
\]
where $X\in \Gamma(TM)$ and $\J\in \Gamma(Q_{\pul})$ satisfies $\J^2=-\id_{TM}$. Then, $\tilde \J$
is the corresponding section of the quaternionic structure $Q_{\hor}$. To check that the almost symplectic form $\omega_{\sf{W}}$ is scalar, it suffices to check that 
\[
\omega_{\sf W}(\tilde \J (X),\tilde \J (Y))=0\,,
\]
for all $X,Y\in \Gamma(TN)$ and $\J\in \Gamma(Q_{\pul}),\J^2=-\id$. It is simple computation that this is equivalent to the claimed condition, because we know the possible irreducible submodules of $[\E\Hh\E\Hh]^*$ (this argument follows by \cite[\modules]{CGWPartI}). \\
The fifth claim occurs directly by the difference $\theta(T(X, Y))+{{\sf P}}(X, Y)-{{\sf P}}(Y, X)$ between the (almost) symplectic structures $\omega_{\sf C}$ and $\omega_{{\sf W},\nabla}$, where $T$ denotes the torsion of $\nabla$. This is because the relation
\[
\theta(T(X, Y))=\theta(T(\J X, \J Y))
\]
can be satisfied only if the torsion vanishes, i.e., $T=0$ (since this component of torsion is not $Q$-Hermitian in these two entries).
Finally, combining the fourth and the fifth claim, the ${{\sf P}}$-parts cancel each other and the sixth claim follows. This completes our proof.
\end{proof}
Let us now illustrate the situation of the above theorem via an example. 
\begin{example}\label{jan2}
We begin with $N=\R^8=(x_1,\dots,x_8)$ and the quaternionic co-frame 
\[
\nu:=(\dd x_1,\dd x_2,\dd x_3,\dd x_4,\dd x_5,\dd x_6,\dd x_7,\dd x_8)\,.
\]
This induces a flat quaternionic structure, and has the following matrix of 1-forms on $N$ as the pullback of the Cartan connection along $\nu$:
\[
C:=\begin{pmatrix}
0&0& 0\\
\dd x_1+\dd x_2 i+\dd x_3 j+\dd x_4 k&0 &0 \\
\dd x_5+\dd x_6 i+\dd x_7 j+\dd x_8 k& 0 & 0
\end{pmatrix}\,.
\]
This selects a flat Oproiu connection $\nabla$. If $\psi: N\to \exp(\fr{p}_+)$ is the function corresponding to the change to another Oproiu connection $\nabla^{\psi}$, then the pullback of the Cartan connection along $\nu\cdot \psi$ becomes
\[
C_\psi:=\Ad(\psi)^{-1}C+\psi^*\mu\,.
\]
Here, $\Ad : G\to\Aut(\fr{g})$ is the adjoint representation of $G={\sf P}\Gl(3,\Hn)$, and $\psi^*\mu$ is the left logarithmic derivative of $\psi: N\to \exp(\fr{p}_+)$, i.e., the pullback of the Maurer-Cartan form
\[
\mu:=(\dd p_1-\dd p_2 i-\dd p_3 j-\dd p_4 k,\dd p_5-\dd p_6 i-\dd p_7 j-\dd p_8 k)
\] to $N$.
Nevertheless, in such a trivialization we obtain
\[
N\times \exp(\fr{p}_+)=\R^8\times \fr{p}_+=(x_1,\dots,x_8,p_1,\dots,p_8)\,,
\] 
and in this case the coordinates $(x_1,\dots,x_8,p_1,\dots,p_8)$ coincide with the coordinates on $T^*N$.
Therefore, similarly to Example \ref{jan1}, we introduce the variables $p,q$, and work with
{\tiny
\begin{gather*}
\begin{pmatrix}
1&-\bar{p}&-\bar{q}\\
0&1 &0 \\
0& 0 & 1
\end{pmatrix}(C_\psi+\dd \exp(\fr{p}_+))\begin{pmatrix}
1&\bar{p} &\bar{q}\\
0&1 &0 \\
0& 0 & 1
\end{pmatrix}
= 
\begin{pmatrix}
*&\dd \bar{r}-(\bar{r}C_{21}+\bar{s}C_{31})\bar{r}& \dd \bar{s}-(\bar{r}C_{21}+\bar{s}C_{31})\bar{s}\\
C_{21}&* &* \\
C_{31}&* &* 
\end{pmatrix}\,,
\end{gather*}}
where $\bar{r}=\bar{p}+\psi_1-\psi_2 i-\psi_3 j-\psi_4 k$, $\bar{s}=\bar{q}+\psi_5-\psi_6 i-\psi_7 j-\psi_8 k$.

Then, according to Theorem \ref{cotangent-example}, the pair $(Q_{\pul},\omega_{W,\nabla})$ provides an almost qs-H structure on $T^*N$, which is induced by the quaternionic Darboux basis by using the following isomorphism:
\[
T(N\times \exp(\fr{p}_+))\to \Hn^4: (C_{21},C_{31},\dd r-(r\bar{C}_{21}+s\bar{C}_{31})r, \dd s-(r\bar{C}_{21}+s\bar{C}_{31})s)\,.
\]
In this basis the quaternionic structure $Q$ corresponds to right multiplication by $-i,-j,-k$, and the scalar 2-form takes the standard form.
Since this almost qs-H structure is locally isomorphic to the qs-H structure from Proposition \ref{quatertangentmodel}, we can conclude that this is a torsion-free qs-H structure.
Moreover, and according to Theorem \ref{cotangent-example}, the pair $(Q_{\hor},\omega_{C})$ is an almost qs-H structure on $T^*N$, which is induced by the quaternionic Darboux basis via the isomorphism 
\[
T(N\times \exp(\fr{p}_+))\to \Hn^4: (C_{21},C_{31},\dd p-2p(r-p)\bar{C}_{21}-(qr+ps-2pq)\bar{C}_{31},\dd q-(qr+ps-2pq)\bar{C}_{21}-2q(s-q)\bar{C}_{31})\,.
\]
Finally, note that ${{\sf P}}_1=\dd(\bar{r}-\bar{p})-((\bar{r}-\bar{p})C_{21}+(\bar{s}-\bar{q})C_{31})(\bar{r}-\bar{p})$ and ${{\sf P}}_2=\dd(\bar{s}-\bar{q})-((\bar{r}-\bar{p})C_{21}+(\bar{s}-\bar{q})C_{31})(\bar{s}-\bar{q})$. Hence, if the tensors ${\sf P}_i$ $(i=1, 2)$ satisfy the conditions of the Theorem \ref{cotangent-example}, then the pair $(Q_{\hor},\omega_{W,\nabla})$ should be an almost qs-H structure on $T^*N$. We avoid presenting the corresponding quaternionic Darboux basis, since it has a long and complicated expression.
\end{example}


\section{Further directions and open problems}\label{opentasks}
In this final section we pose some questions and open problems related to the geometry of $\SO^*(2n)$- and $\SO^*(2n)\Sp(1)$-structures. Some of them are studied in the third part of this series of works. 
\begin{itemize}
\item[${\sf I)}$] Examine curvature invariants, Bianchi identities, Ricci-type and other curvature types of $\SO^*(2n)$- and $\SO^*(2n)\Sp(1)$-structures. 
\item[${\sf II)}$] Realization of the pure algebraic types $\mc{X}_{1}, \ldots, \mc{X}_5$ of $\SO^*(2n)\Sp(1)$-structures, as well of the $\Sp(1)$-invariant types $\mc{X}_{1}, \ldots, \mc{X}_7$ of $\SO^*(2n)$-structures, specified in \cite{CGWPartI}. Are there any empty classes? Provide a classification of manifolds having some certain algebraic type (e.g. skew-torsion, vectorial type), or certain holonomy with respect to $\nabla^{H, \omega}$ or $\nabla^{Q, \omega}$, respectively. 
\item[${\sf III)}$] Provide a metric view point of $\SO^*(2n)$-structures: Define such structures in terms of the three metrics of signature $(2n, 2n)$, and application of the approach discussed in Section \ref{sec2II} for the Levi-Civita connection corresponding to one of these metrics. Study the associated pseudo-Riemannian Dirac operators, see also problem ${\sf X)}$.
\item[${\sf IV)}$] Description of adapted twistor constructions (as an analogue of the twistor constructions in almost hH/qH geometries, see \cite{Salamon82, Swann91}). Relate the given integrability conditions to such a description. 
\item[${\sf V)}$] Provide the characterization of the quaternionic skew-Hermitian geometries in the image of the functor from Theorem \ref{weylstr}. 
\item[${\sf VI)}$] From the parabolic view point (see \cite{Cap14}) it makes sense to derive the explicit relations with quaternionic geometries. That is, analyze the $\SO^*(2n+2)$-orbits on the quaternionic projective space $\Hn{\sf P}^{n}=\Gl(n+1, \Hn)/P$ (note that the open orbit is the symmetric space $\SO^*(2n+2)/\SO^*(2n)\U(1)$ discussed in \cite[Section 5]{CGWPartI}), and interpret the parallel tractor scalar 2-forms. Investigate the corresponding first BGG operators and interpret the normal solutions in terms of a generalized Einstein condition.
\item[${\sf VII)}$] Quaternionic compactification of $\SO^*(2n)\U(1)$-structures and examination of the geometry on the boundary. In particular, what is the type of the parabolic geometry on the boundary?
\item[${\sf VIII)}$] Examine almost quaternion skew-Hermitian manifolds with large automorphism group, and
classify those for which the dimension of the group is close to the maximal dimension. 
\item[${\sf IX)}$] Classify non-symmetric homogeneous spaces $G/L$ admitting an invariant almost hypercomplex/quaternionic skew-Hermitian structure, with $G$ semisimple and $L$ reductive.
\item[${\sf X)}$] Provide a spinorial interpretation of $\SO^*(2n)$-structures and $\SO^*(2n)\Sp(1)$-structures, and introduce the associated spinorial calculus and adapted Dirac operators. Study metaplectic geometry and symplectic Dirac operators related with metaplectic structures on $8n$-dimensional quaternionic skew-Hermitian manifolds.
\end{itemize}

\appendix

\section{Topology of $\SO^*(2n)$- and $\SO^*(2n)\Sp(1)$-structures}\label{appendix2}
In this appendix we study some basic topological features of $\SO^*(2n)$- and $\SO^*(2n)\Sp(1)$-structures. To do so, we should recall first some covering theory related to $\SO^*(2n)$.
There are several distinguished Lie groups with Lie algebra $\so^*(2n)$, apart from $\SO^{*}(2n)$. Firstly, there is \textsf{the universal covering} $\widetilde{\SO}^{*}(2n)$ of $\SO^{*}(2n)$. This gives rise to the following short exact sequence
\[
0\longrightarrow \Z\longrightarrow\widetilde{\SO}^*(2n)\longrightarrow\SO^*(2n)\longrightarrow 0\,.
\]
Note that any representation of $\so^*(2n)$ integrates to representation of $\widetilde{\SO}^{*}(2n)$, which however for the finite dimensional case is {\it not} faithful. In other words, $\widetilde{\SO}^{*}(2n)$ has no finite-dimensional faithful representations. 
At this point we should however recall that there is a \textsf{maximal linear group} attached to the Lie algebra $\fr{so}^{*}(2n)$, also called the \textsf{linearizer} of $\widetilde{\SO}^{*}(2n)$ and denoted by $\Spin^{\ast}(2n)$ (see e.g. \cite[pp.~264, 320]{OnV} for this notion).\footnote{In \cite{OnV} the group $\SO^*(2n)$ is denoted by $\U^*_{n}(\Hn)$, see page 226.} This is characterized by the following universal property: 
\par {\it Any finite dimensional linear representation $\rho : \widetilde{\SO}^*(2n)\to\Aut(W)$ of $\widetilde{\SO}^*(2n)$ factors as $\rho = \rho_0 \circ \psi$, where $\psi :\widetilde{\SO}^*(2n)\to \Spin^{\ast}(2n)$ is a (covering) homomorphism and $\rho_0 : \Spin^{\ast}(2n)\to \Aut(W)$ is a linear representation of $\Spin^{\ast}(2n)$. That is, the following diagram commutes:}
\[
\xymatrix{
\widetilde{\SO}^*(2n) \ \ar[r]^{\psi} \ar[dr]_{\rho} & \Spin^*(2n) \ar[d]^{\rho_0} \\
& \Aut(W)\,.
}
\]
Based on the embedding $\SO^*(2n)\subset\SO(2n, 2n)$ one should view $\Spin^*(2n)$ as a subgroup of $\Spin(2n, 2n)$. In particular, 
the faithful representation of $\Spin^{*}(2n)$ is given by the direct sum 
\[
R(\pi_{n-1})\oplus R(\pi_{n}) 
\]
of the (finite-dimensional) half-spin representations of $\fr{so}^*(2n)$, which is the reason behind our notational convention for $\Spin^*(2n)$. 
In particular, the group $\Spin^*(2n)$ is a (double) covering of $\SO^{*}(2n)$, which induced the following short exact sequence
\[
0\longrightarrow \Z_2\longrightarrow\Spin^*(2n)\overset{\lambda}{\longrightarrow}\SO^*(2n)\longrightarrow 0\,,
\]
such that $\lambda_{*}^{-1}(\fr{spin}^{*}(2n))=\fr{so}^*(2n)$. 

In small dimensions $n\leq 4$ one may describe further Lie algebras isomorphisms, which we list below and which can be easily interpreted in terms of Satake diagrams (see \cite{Hel, Harvey, OnV} for more details).
For $n=1$, $\SO^{\ast}(2)$ is isomorphic to $\SO(2)=\U(1)$, hence it is compact and non-simple.
For $n=2, 3, 4$, the corresponding Satake diagrams admit the following illustration (see \cite{Oni, CGWPartI})
\[
\begin{picture}(30,15)(-8,-12)
\put(0, 0){\circle*{4}}
\put(0,8){\makebox(0,0){{\tiny$\Lambda_1$}}}
\put(18, 0){\circle{4}}
\put(18.5,8){\makebox(0,0){{\tiny$\Lambda_2$}}}
\end{picture}
\quad
\begin{picture}(60, 20)(-20, -12)
\put(50, 1){\line(2,1){10.5}}
\put(50, -1){\line(2,-1){10.5}}
\put(48, 0){\circle*{4}}
\put(48,-7){\makebox(0,0){{\tiny$\Lambda_{1}$}}}
\put(62.5, -7){\circle{4}}
\put(70, 13){\makebox(0,0){{\tiny$\Lambda_{2}$}}}
\put(66.5, -10){{\tiny$\Lambda_{3}$}}
\put(62.5, 7){\circle{4}}
\put(63, -5){\vector(0,1){9.5}}
\put(63, -4){\vector(0, -1){1.5}}
\end{picture}\quad 
\begin{picture}(120,34)(-15,-12)
\put(82, 0){\circle*{4}}
\put(84, 0){\line(1,0){12}}
\put(100, 1){\line(2,1){10.5}}
\put(100, -1){\line(2,-1){11}}
\put(98, 0){\circle{4}}
\put(98,-7){\makebox(0,0){{\tiny$\Lambda_{2}$}}}
\put(85.5, 8){\makebox(0,0){{\tiny{$\Lambda_{1}$}}}}
\put(112.5, -7){\circle{4}}
\put(120, 13){\makebox(0,0){{\tiny$\Lambda_{3}$}}}
\put(111, -16){{\tiny$\Lambda_{4}$}}
\put(112.5, 7){\circle*{4}}
\end{picture} 
\]
Hence:
\begin{itemize}
\item For $n=2$ there is a Lie algebra isomorphism $\fr{so}^{\ast}(4)\cong\fr{su}(2)\oplus\fr{sl}(2, \R)$. Moreover, $\Spin^*(4)$ coincides with the Lie group $\SU(2) \times \Sl(2, \R)$ and we have a double covering
$
\SU(2) \times \Sl(2, \R) \to \SO^{\ast}(4).
$
Hence $\SO^*(4)$ is semisimple and non-simple.
\item For $n=3$, there is a Lie algebra isomorphism $\fr{so}^{\ast}(6)\cong\fr{su}(1,3)$. Since $Z(\SO^{\ast}(6))=\Z_{2}$ and $Z(\SU(1, 3))=\Z_{4}$, we get that $\Spin^*(6)$ coincides with the Lie group $\SU(1,3)$ and the map
$
\SU(1,3) \to \SO^{\ast}(6) 
$ defines a double covering.
\item For $n=4$, there is a Lie algebra isomorphism $\fr{so}^{\ast}(8)\cong\fr{so}(2,6)$. The half-spin groups associated to $\SO(2,6)$ are isomorphic to $\SO^\ast(8)$.
On the other side, $\Spin^*(8)$ and $\Spin(2, 6)$ are both maximal linear groups of $\fr{so}^*(8)$ and hence are identical. This means that there are covering homomorphisms $\Spin(2,6) \rightarrow \SO^\ast(8) \rightarrow P\SO(2,6)$, where the latter group is the one which acts faithfully on the projectivization of $\R^8$. In particular, $Z(\Spin^*(8))=\Z_2\times \Z_2$ (see \cite[p.~320]{OnV}). 
However, $\SO(2,6)$ is not isomorphic to $\SO^\ast(8)$ as Lie groups, since for example the have different maximal compact subgroups, i.e. $\SO(2)\SO(6)\ncong\U(4)$. Another argument is coming from topology: $\SO^*(8)$ is connected but $\SO(2,6)$ has two connected components. 
\end{itemize}

Next we discuss topological obstructions related to $\SO^*(2n)$-structures and $\SO^*(2n)\Sp(1)$-structures. 
We begin with the maximal compact subgroup $\U(n)$ of $\SO^*(2n)$, which can be viewed as the following block diagonal matrix
\[
\left\{\begin{pmatrix}
A & 0 \\
0 & A
\end{pmatrix} : A\in\U(n)\right\}\subset\U(2n)\,,
\]
where $\U(2n)$ is the maximal compact subgroup of $\Sp(4n,\R)$. This shows that $\U(n)$ does not act irreducibly on $\E=\C^{2n}$. However, $\Sp(1)$ coincides with the centralizer $C_{\U(2n)}(\U(n))$ and hence also lies inside $\U(2n)\subset\Sp(4n,\R)$. Thus, we we finally obtain the following necessary topological conditions, naturally arising within the theory of $\U(k)$-structures (see \cite[p.~200]{Liber} or \cite[p.~359]{CE14} and recall also that according to \cite[Problem 14-B]{Milnor}, the odd-degree Stiefel-Whitney classes of a complex vector bundle must vanish).
\begin{lem}
Let $M$ be a smooth $4n$-dimensional manifold admitting a $\SO^{\ast}(2n)\Sp(1)$-structure or a $\SO^{\ast}(2n)$-structure. Then all odd Stiefel-Whitney class $w_{2k+1}(M)\in H^{2k+1}(M; \Z_{2})$ must vanish, $w_{2k+1}(M)=0$, for any $k\in \mathbb{N}$, and the even one must have integral lifts.
\end{lem}
Thus, for instance, such manifolds should be oriented, that is $w_{1}(M)=0$, a fact which agrees with our observation in \cite{CGWPartI} and the orientation constructed via the volume form $\vol=\omega^{2n}$. Note however that the above condition is only necessary and not sufficient: For instance the tangent bundle $T\Ss^n$ of $\Ss^n$ always satisfies these conditions, but only $\Ss^2$ and $\Ss^6$ admit an almost complex structure, and in particular a $\U(n)$-structure. For the compact case one can pose further topological constraints related to the existence of an almost complex structure, which can be read for example in terms of the Euler characteristic of $M$ (see Theorem 3.4 in the appendix of \cite{Devito}).

\smallskip
Let us now fix an almost qs-H manifold $(M, Q, \omega)$ $(n>1)$, and denote by $\pi : \mc{Q}\to M$ the corresponding principal $\SO^*(2n)\Sp(1)$-bundle over $M$. Then, any point in $\mc{Q}$ provides an identification between $T_{x}M$ and $[\E\Hh]$, for any $x\in M$. Next we study the lifting problem of the $G=\SO^*(2n)\Sp(1)$-structure to an $\tilde{G}=\SO^*(2n)\times\Sp(1)$-structure that allows to associate global bundle analogies of the modules $\E$ and $\Hh$ over $M$. 

To do so, we view $\mc{Q}$ as an element of the first \v{C}ech cohomology group $H^{1}(M; G)$ with coefficients in the sheaf of smooth $G$-valued functions on $M$. Moreover, $\tilde{G}$ is the double cover of $G$ and so the short exact sequence 
\[
0\longrightarrow \Z_2 \longrightarrow \tilde{G}\longrightarrow G\to 0
\]
induces the coboundary homomorphism 
\[
\updelta : H^{1}(M; G)\to H^{2}(M; \Z_2)\,.
\]
By setting $\upepsilon:=\updelta(\mc{Q})\in H^{2}(M; \Z_2)$ 
we obtain a canonical cohomology class on $M$, which we will call the \textsf{Marchiafava-Romani class}. Then, as an analogue of the almost quaternionic-Hermitian case (see \cite{MR, Salamon82, Salamon86}), we deduce the following.
\begin{lem}\label{analog1}
The Marchiafava-Romani class $\upepsilon\in H^{2}(M; \Z_2)$ is precisely the obstruction to lifting $\mc{Q}$ to a principal $\tilde{G}$-bundle $\tilde{\mc{Q}}$ over $M$, or equivalently to the global existence of the vector bundles $E=\mc{Q}\times_{G}\E$ and $H=\mc{Q}\times_{G}\Hh$ over $M$. 
\end{lem}
Thus, when $\upepsilon=0$, the bundles $E$ and $H$ are globally defined and thus
\[
T^{\C}M\cong E\otimes H\,.
\] 
Recall now that given a pair $(Q, \omega)$ as above, we may visualise the almost quaternionic structure $Q\subset\Ed(TM)$ via the coefficient bundle associated with the $\SO^*(2n)\Sp(1)$-representation $[S^2\Hh]^*$ (see for example \cite{Salamon86} or \cite[\invarena]{CGWPartI}). Thus, one can identify $\upepsilon$ with the second Stiefel-Whitney class of $Q$, i.e.
$\upepsilon=w_{2}(Q)$, see \cite{MR, Salamon82}. Moreover, in terms of the twistor bundle $Z=\mathbb{P}(\Hh)\to M$ which is naturally associated to the almost quaternionic structure $Q$, one obtains the following
\begin{prop}\label{saltop}
Let $(M, Q, \omega)$ be an almost qs-H manifold. Then, the second Stiefel-Whitney class $w_{2}(M)\in H^{2}(M; \Z_2)$ satisfies
\[
w_{2}(M)=
\left\{
\begin{tabular}{lc}
$\upepsilon$, & if \ $n=\text{odd}$,\\ 
$0$, & if \ $n=\text{even}$.
\end{tabular}\right.
\]
\end{prop}
\begin{proof}
The proof is the same with the one given in \cite{Salamon82}, for quaternionic K\"ahler manifolds, although the same result and proof applies for general almost quaternion Hermitian manifolds.
\end{proof}

As is well-known, on a $4n$-dimensional oriented manifold $M$ the vanishing of the second Stiefel-Whitney class $w_2$ guarantees the existence of \textsf{spin structures} (see for example \cite{LM}) and also of \textsf{metaplectic structures} (see for example \cite{Hab}).
	Thus, for $8n$-dimensional almost qs-H manifolds, Proposition \ref{saltop} certifies that such a manifold should admit these types of structures. In particular, by the inclusion $\SO^*(2n)\Sp(1)\subset\Sp(4n, \R)$ we conclude that the vanishing of $w_2(M)$ for $n=2m$, guarantees the reduction of the metaplectic structure to a certain 2-fold covering of $\SO^*(2n)\Sp(1)$. Such a (unique) 2-fold covering is given by $\Spin^*(2n)\Sp(1)$, and the corresponding lifts of the $G$-structure can been seen as a generalization of the so-called $\Spin^{q}$-structures, discussed in \cite{Nag95, Bar, Alben}.
We plan to examine $\Spin^*(2n)\Sp(1)$-structures on $8n$-dimensional almost qs-H manifolds $(M, Q, \omega)$ in a forthcoming paper in this series.

\end{document}